\documentclass{article}

\usepackage{a4} 
\usepackage{amsmath}
\usepackage{amssymb}
\usepackage{amsthm}
\usepackage{graphicx}
\usepackage{caption,subcaption}
\usepackage{multirow}
\usepackage{xstring}
\usepackage[utf8]{inputenc}
\usepackage[american]{babel}
\usepackage{amstext,amsbsy,amsopn,eucal,enumerate}
\usepackage{tikz}
\usepackage{pgfplots}
\usepackage{tabularx}
\title{Balanced truncation and singular perturbation approximation model order reduction for stochastically controlled linear systems}
\author{Martin Redmann
        \thanks{Weierstrass Institute for Applied Analysis and Stochastics, Mohrenstrasse 39, 10117 Berlin
Germany, {\tt martin.redmann@wias-berlin.de}.} 
\and Melina A. Freitag
	    \thanks{Department of Mathematical Sciences, University of Bath, Claverton Down, BA2 7AY, United Kingdom, {\tt m.a.freitag@bath.ac.uk}.}
}

\def\R{\mathbb{R}}

\newcommand{\expn}{\operatorname{e}}

\newcommand{\diag}{\operatorname{diag}}
\newcommand{\spaned}{\operatorname{span}}
\newcommand{\kernel}{\operatorname{ker}}

\newcommand{\im}{\operatorname{im}}

\newcommand{\beq}{\begin{equation}}
\newcommand{\eeq}{\end{equation}}
\newcommand {\mat}      [1] {\left[\begin{array}{#1}}
\newcommand {\rix}          {\end{array}\right]}
\newcommand{\trace}{\operatorname{tr}}

\newtheorem{defn}{Definition}[section]

\newtheorem{example}[defn]{Example}
\newtheorem{lem}[defn]{Lemma}
\newtheorem{prop}[defn]{Proposition} 
\newtheorem{kor}[defn]{Corollary}
\newtheorem{thm}[defn]{Theorem}

\begin{document}
\maketitle

\begin{abstract}
When solving linear stochastic differential equations numerically, usually a high order spatial discretisation is used.
Balanced truncation (BT) and singular perturbation approximation (SPA) are well-known projection techniques in the deterministic
framework which reduce the order of a control system and hence reduce computational complexity. This work considers
both methods when the control is replaced by a noise term. We provide theoretical tools such as stochastic 
concepts for reachability and observability, which are necessary for balancing related model order reduction of
linear stochastic  differential equations with additive L\'evy noise. Moreover, we derive error bounds for both BT
and SPA and provide numerical results for a specific example which support the theory.
\end{abstract}

\paragraph{Keywords:} model order reduction, balanced truncation, singular perturbation approximation, stochastic systems, L\'evy process, Gramians, Lyapunov equations.

\paragraph{AMS Subject Classifications}
Primary 93A15, 93B40, 93E03, 93E30, 60J75. Secondary 93A30, 15A24.


\section{Introduction}

Many mathematical models of real-life processes pose challenges during numerical 
computations, due to their large size and complexity. Model order reduction (MOR) techniques are methods that reduce the computational complexity of numerical 
simulations, an overview of MOR methods is provided in \cite{antoulas,schilders2008model}. MOR techniques such as balanced truncation (BT) and singular perturbation approximation
(SPA) are methods which have been introduced in \cite{moore} and \cite{spa}, respectively, for linear deterministic systems 
\[
\dot{x}(t) = Ax(t)+Bu(t),\quad y(t) = Cx(t).
\]
Here $A\in\mathbb{R}^{n\times n}$ is asymptotically stable, $B\in \mathbb R^{n\times m}$, $C\in 
\mathbb R^{p\times n}$ and $x(t)\in\mathbb{R}^n$, $y(t)\in\mathbb{R}^p$, $u(t)\in\mathbb{R}^m$ are state, output and input of the system, respectively. 
From the Gramians $P$ and $Q$ which solve dual Lyapunov equations 
\[
AP+PA^T= -BB^T, \quad A^T Q+QA = -C^TC,
\]
a balancing transformation is found, which is used to project the state space of size $n$ to a much smaller dimensional state space (see, e.g.  \cite{antoulas}). 

Recently, the theory for BT and SPA has been extended to stochastic linear systems of the form
\begin{equation}
\label{linsysafterdisintr}
dx(t) = Ax(t) dt+Bu(t)dt+\sum_{k=1}^q N_k x(t-) dM_k(t), \quad y(t) = Cx(t),
\end{equation}
where $A$, $B$ and $C$ as above, and $N_k\in\mathbb{R}^{n\times n}$ and $M_k$ ($k= 1, \ldots, q$) are 
uncorrelated scalar square integrable L\'evy processes with mean zero (often 
$q=1$ and the special case of Wiener processes are considered, see, for example, 
\cite{bennerdamm,dammbennernewansatz,hartmann2013balanced}). In this case BT and SPA require the solution of more general Lyapunov equations of the form 
\[
AP+PA^T+\sum_{k=1}^q N_k P N_k^T c_k = -BB^T, \quad A^T Q+QA +\sum_{k=1}^q N_k^T Q N_k c_k= -C^TC,
\]
where $c_k = \mathbb{E}[M_k(1)^2]$ for general L\'evy processes. Note that $c_k=1$, $k=1,\ldots, q$ for the case of a Wiener process \cite{bennerdamm}.
We refer to \cite{redmannbenner,dammbennernewansatz,redmann16, redSPA} for a detailed theoretical and numerical treatment of balancing related MOR for (\ref{linsysafterdisintr}). 

In this paper we are going to study balancing related MOR for systems of the form
\begin{equation}
\label{linsysafterdisintr2}
dx(t)=A x(t)dt+ BdM(t),\quad y(t)= C x(t),
\end{equation}
where $BdM(t) = \sum_{i=1}^m b_i dM_i(t)$ and $b_i$ is the $ith$ column of $B\in\mathbb{R}^{n\times m}$. The processes $M_i$ are 
the components of a square integrable mean zero L\'evy process $M=\left(M_1, \ldots, M_m\right)^T$ that takes values in $\mathbb R^m$. Consequently, these 
components are not necessarily uncorrelated. For a general theoretical treatment of SDEs with L\'evy noise we refer to \cite{applebaumendlich}.

The setting in (\ref{linsysafterdisintr2}) is of particular interest in many applications. If one is interested in a large number of different realisations of the output
$y(t)$ (e.g. to compute moments of the form $\mathbb E\left[ f(y(t))\right]$), then one needs to solve the SDE in (\ref{linsysafterdisintr2}) a large number of times.
For a state space of high dimension this is computationally expensive. Reduction of the state space dimension decreases the
computational complexity when sampling the solution to (\ref{linsysafterdisintr2}), as the SDE can then be solved in much smaller dimensions. Hence the computational costs are reduced dramatically.

The linear system (\ref{linsysafterdisintr2}) is a problem where the control is noise. In this case the standard theory for balancing related MOR applied to a deterministic system no longer applies. 

Balanced truncation has been applied to linear systems with white noise before. The discrete time setting was discussed in \cite{ArunKung90}. For the continuous time setting, dissipative Hamiltonian systems with Wiener noise were treated in \cite{hartmann2011,hartmann2008}, but no error bounds were provided. In this paper we consider both BT and SPA model order reduction. As far as we are aware, no theory and in particular error bounds for balancing related MOR have been developed for continuous time SDEs with L\'evy noise.

Using theory for linear stochastic differential equations with additive L\'evy noise we provide a stochastic concept of reachability. This concept motivates a new formulation of the reachability Gramian. We prove bounds for the error between the full and reduced system which provide criteria for truncating, e.g. criteria for a suitable size of the reduced system. We analyse both BT and SPA and apply the theory  directly to an application arising from a second order damped wave equation. 

We now consider a particular example which explains why the above setting is of practical interest.

\paragraph{Motivational example}
In \cite{redbensec} the lateral time-dependent displacement $\mathcal Z$ of an electricity cable impacted by wind was modeled by the following one-dimensional symbolic second order SPDE with L\'evy 
noise:
\begin{align}\label{stringexpintro}
 \frac{\partial^2}{\partial t^2} \mathcal Z(t, \zeta)+\alpha 
\frac{\partial}{\partial t} 
 \mathcal Z(t, \zeta)=\frac{\partial^2}{\partial \zeta^2} 
\mathcal Z(t,\zeta)+\expn^{-(\zeta-\frac{\pi}{2})^2}u(t)+2\expn^{-(\zeta-\frac{
\pi}{2})^2 }\mathcal Z(t-, \zeta) \frac{\partial}{\partial t} M_1(t)
\end{align}
for $t\in[0, T]$, $\zeta\in[0, \pi]$ and $\alpha>0$, with boundary and initial conditions 
\begin{align}
\label{eq:bcic}
\mathcal Z(0, t)=0=\mathcal Z(\pi, t)\;\;\;\text{and}\;\;\; \mathcal 
Z(0,\zeta), \left.\frac{\partial}{\partial 
t} \mathcal Z(t, \zeta)\right\vert_{t=0}\equiv 0.
\end{align}
For small $\epsilon>0$, the output equation 
\begin{align}\label{introbspout}
\mathcal Y(t)=\frac{1}{2\epsilon}\int_{\frac{\pi}{2}-\epsilon}^{\frac{\pi}{2}
+\epsilon }\mathcal Z(t,\zeta) d\zeta
\end{align}
is approximately the position of the middle of the cable.
In \cite{redbensec}, it is shown that transforming this SPDE in 
into a first order SPDE and then discretising it 
in space, leads to a system of the form (\ref{linsysafterdisintr}) where $q=m=p=1$. 

One drawback of the approach above is, that, when the electricity cable is in steady state, the wind has no impact. A more realistic scenario, which models the wind as some form of stochastic input, is the following symbolic equation
\begin{align}\label{inrobspeq}
 \frac{\partial^2}{\partial t^2} \mathcal Z(t, \zeta)+\alpha 
\frac{\partial}{\partial t} 
 \mathcal Z(t, \zeta)=\frac{\partial^2}{\partial \zeta^2} 
\mathcal Z(t,\zeta)+\sum_{k=1}^m f_k(\zeta)\frac{\partial}{\partial t} M_k(t)
\end{align}
for $t\in[0, T]$, $\zeta\in[0, \pi]$ and $\alpha>0$, boundary and initial conditions as in (\ref{eq:bcic}), and $M_k$ the components of a square integrable mean zero L\'evy process $M=\left(M_1, \ldots, M_m\right)^T$ that takes values in $\mathbb R^m$. In this paper, we consider a framework which covers this model. Moreover we modify the output in (\ref{introbspout}) and let
\begin{align}\label{introbspoutnew}
\mathcal 
Y(t)=\frac{1}{2\epsilon}\left(\begin{matrix}\int_{\frac{\pi}{2}-\epsilon}^{
\frac{\pi}{2}+\epsilon }\mathcal Z(t,\zeta) d\zeta &
\int_{\frac{\pi}{2}-\epsilon}^{\frac{\pi}{2}
+\epsilon }\frac{\partial}{\partial t} \mathcal Z(t,\zeta) 
d\zeta\end{matrix}\right)^T,
                  \end{align}
so that both the position and velocity of the middle of the string are observed. Transformation and discretisation of this SPDE leads to a system of the form (\ref{linsysafterdisintr2}) where $A$ is an asymptotically stable matrix, i.e. $\sigma(A)\subset \mathbb C_-$.

This paper is set up as follows. Section \ref{sec:balancing} provides the theoretical tools for balancing linear SDEs with additive
L\'evy noise. We explain the theoretical concepts of reachability and observability in this setting and show how this motivates MOR using
BT and SPA. Moreover we provide theoretical error bounds for both methods. In Section \ref{sec:wave} we show how a wave equation driven by L\'evy noise can
be transformed into a first order equation and then reduced to a system of the form (\ref{linsysafterdisintr2}) by using a spectral Galerkin method.
Numerical results which support our theory are provided in Section \ref{sec:numerics}. 

\section{Balancing for linear stochastic differential equations with additive L\'evy noise}
\label{sec:balancing}

In \cite{antoulas, spa, moore} balancing related MOR was considered for deterministic systems of the form 
\begin{align}\label{detsysant}
             \dot x(t)&=Ax(t)+Bu(t),\;\;\;x(0)=0,\\ \nonumber
             y(t)&=Cx(t), \;\;\;t\geq 0,
            \end{align}
where $A\in \mathbb R^{n\times n}$ was assumed to be asymptotically stable, i.e. $\sigma(A)\subset \mathbb C_-$, $B\in \mathbb R^{n\times m}$, $C\in 
\mathbb R^{p\times n}$ and $u\in L^2([0, T])$ for all $T>0$ was a deterministic control. 

We now turn our attention to a stochastic system 
\begin{align}\label{stochsysnew}
             dx(t)&=Ax(t)dt+BdM(t),\;\;\;x(0)=x_0\\ \nonumber
             y(t)&=Cx(t), \;\;\;t\geq 0,
            \end{align}
which, in Section \ref{subsecspeltralgal}, represents a spatially discretised version of an SPDE. The matrices $A$, $B$ and $C$ are as above and the 
$\mathbb R^m$-valued process $M$ is a square integrable 
L\'evy processes with mean zero. One might interpret system 
(\ref{stochsysnew}) as system (\ref{detsysant}) with $u(t)=\dot M(t)$ but the noise $\dot M(t)$ is no control in the classical sense. First of all, stochastic controls were not admissible
in the deterministic setting and secondly the classical derivative does not exist. So, if we want to study balancing related MOR for the ``particular control'' $\dot 
M(t)$, we need to make sense of this setting which we do by the Ito-type SDE in (\ref{stochsysnew}).

In the deterministic case reachability and observability concepts are 
introduced to characterise the importance of states. Difficult to reach states (states 
which require large energy to reach them) and difficult to observe states 
(states which only produce little observation energy) are seen to be 
unimportant in the systems dynamics. In balancing related MOR 
the idea is to create a system, where the dominant reachable and observable 
states are the same. Those are then truncated to obtain a reduced order model (ROM).

Applying balancing related MOR to (\ref{stochsysnew}) requires a few modifications compared to the classical deterministic framework. We introduce a stochastic reachability concept in Section \ref{reachsec} which 
also leads to a different reachability Gramian compared to the deterministic 
case. For the observation concept we follow 
the deterministic approach. We then describe the procedure of balancing for 
systems with additive noise in Section \ref{subsecprocedure} which is similar 
to the deterministic case. Afterwards, we will discuss two particular techniques 
which are BT and SPA. Since it is not a priori clear whether these approaches for system 
(\ref{stochsysnew}) perform as well as for deterministic systems, 
we contribute an error bound for both BT and SPA in Section \ref{btandspa}. 
These error bounds enable us to point out the cases, where BT and SPA work well, and they can be used to find a suitable ROM dimension.

\subsection{Reachability and Observability}\label{reachsec}

With suitable reachability and observability concepts we want to analyze which 
states in system (\ref{stochsysnew}) are unimportant and hence can be 
neglected. 

\paragraph{Reachability} We begin with a stochastic reachability concept, where the 
particular choice of $M$ is taken into account. Starting from zero ($x_0=0$) in 
\begin{align}\label{stochstatenew}
             dx(t)=Ax(t)dt+BdM(t),\;\;\;x(0)=x_0, \;\;\;t\geq 0,
            \end{align}
we investigate how much the noise can control the state away from zero. We define what is meant be reachability in the stochastic case, 
where $x(t, x_0, M)$, $t\geq 0$, denotes the solution to (\ref{stochstatenew}) 
with initial condition $x_0\in\mathbb R^n$ and noise process $M$.
\begin{defn}
 A state $x\in\mathbb R^n$ is not reachable from zero on the time interval $[0, 
T]$, $T>0$, if it is contained in an open set $O$ with \begin{align*}
      \mathbb P\left\{x(t, 0, M)\in O,\;\;\;\text{for every }t\in[0, 
T]\right\}=0,
                                                       \end{align*}
else $x$ is reachable. The system is called completely reachable if 
\begin{align}\label{completereach}
      \mathbb P\left\{x(t, 0, M)\in O,\;\;\;\text{for some }t\in[0, T]\right\}>0
                                                       \end{align}
for every open set $O\subseteq\mathbb R^n$. 
\end{defn}

\noindent We refer to \cite{stochreach}, where weak controllability was analyzed for 
equations with Wiener noise. Weak controllability turns out to be similar to 
condition (\ref{completereach}).\smallskip

To characterise the degree of reachability of a state, we introduce finite time 
reachability Gramians $P(t):=\mathbb{E} \left[x(t, 0, M)x^T(t, 0, M)\right]$ 
which are the covariance matrices of $x(t, 0, M)$ at fixed times $t\geq 0$. 
Before we study the meaning of these Gramians, we show that $P(t)$ is the 
solution of a matrix differential equation.
\begin{prop}
 The matrix-valued function $P(t)$, $t\geq 0$, is the solution to 
\begin{align}\label{matrixdiffgl}
\dot X(t)=AX(t)+X(t)A^T+ B \mathcal Q_M B^T,
    \end{align}
where $\mathcal Q_M=\mathbb E[M(1) M^T(1)]$ is the covariance matrix of $M$ at 
time $1$.
\end{prop}
\begin{proof}
We replace $x(t, 0, M)$ by $x(t)$ to shorten the notation in the proof. Using 
Ito's formula in Corollary \ref{iotprodformelmatpro}, we obtain the following 
for $x(t)x^T(t)$, $t\geq 0$:
\begin{align*}
x(t) x^T(t)=\int_0^t x(s-) dx^T(s)+\int_0^t dx(s) 
x^T(s-)+\left([e_i^T x,e_j^T x]_t\right)_{i, j=1, \ldots, n},
\end{align*}
where $e_i$ is the $i$-th unit vector and we used $x_0=0$. Inserting the stochastic differential of 
$x(t)$ yields \begin{align*}
&\int_0^t x(s-) dx^T(s)=\int_0^t x(s-) x^T(s) A^T ds+ \int_0^t x(s-)dM^T(s) B^T 
\;\;\;\text{and}\\
&\int_0^t dx(s) x^T(s-)=\int_0^t A x(s) x^T(s-)ds+\int_0^t 
B dM(s) x^T(s-).
\end{align*}
Since the Ito integrals have mean zero, we have  \begin{align*}
\mathbb E\left[x(t) x^T(t)\right]=\int_0^t \mathbb 
E\left[x(s) x^T(s)\right] A^T ds+\int_0^t A \mathbb 
E\left[x(s) x^T(s)\right] ds+\left(\mathbb E[e_i^T 
x, e_j^T x]_t\right)_{i, j=1, \ldots, n},
\end{align*}
where we replaced $x(s-)$ by $x(s)$. This does not impact the 
integrals since a c\`adl\`ag process has at most countably many jumps on a 
finite time interval (see \cite[Theorem 2.7.1]{applebaumendlich}). 
Applying Corollary \ref{iotprodformelmatpro} again, the 
stochastic differential of $BM(t)M^T(t)B^T$ is given by: 
\begin{align*}
BM(t)(BM(t))^T=&\int_0^t BM(s-) dM^T(s) B^T+B\int_0^t dM(s) 
(BM(s-))^T\\&+\left([e_i^T BM,e_j^T BM]_t\right)_{i, j=1, \ldots, n},
\end{align*}
Taking the expected value, we have $
B\mathbb E[M(t)M^T(t)]B^T=\mathbb E\left([e_i^T BM,e_j^TBM]_t\right)_{i, 
j=1, \ldots, n}$. In \cite[Theorem 4.44]{zabczyk} it was shown that the covariance function is 
linear in $t$, i.e. $\mathbb E[M(t)M^T(t)]=\mathcal Q_M t$. Since the $i$th 
component \begin{align*}
    e_i^Tx(t)=e_i^T\int_0^t Ax(s)ds+e_i^T BM(t),\;\;\;t\geq 0,
 \end{align*}
has the same jumps and the same martingale part as $e_i^T BM$, we know by 
(\ref{decomqucov}) that $[e_i^T BM,e_j^TBM]_t=[e_i^T x, e_j^T x]_t$ for $i, j=1, \ldots, n$. Summarizing the results, we have 
\begin{align*}
\mathbb E\left[x(t) x^T(t)\right]=\int_0^t \mathbb 
E\left[x(s) x^T(s)\right] A^T ds+\int_0^t A \mathbb 
E\left[x(s) x^T(s)\right] ds+B\mathcal Q_M B^T t,
\end{align*}
which concludes the proof.
\end{proof}

\noindent To find a representation for $P(t)$ we need the following straightforward result.
\begin{prop}\label{prodeinfachh}
Let $A_i\in\mathbb R^{d_i\times d_i}$ and $K_i\in \mathbb R^{d_i\times m}$, then \begin{align*}
\expn^{A_1t}K_1 K_2^T \expn^{A_2^T t}=K_1 K_2^T+A_1 \int_0^t\expn^{A_1s}K_1 
K_2^T \expn^{A_2^Ts}ds+ \int_0^t\expn^{A_1s}K_1 K_2^T \expn^{A_2^Ts}ds A_2^T.
\end{align*}
\end{prop}
\begin{proof}
 The product rule yields \begin{align*}
d\left(\expn^{A_1t}K_1 K_2^T \expn^{A_2^Tt}\right)=A_1 \expn^{A_1t}K_1 K_2^T 
\expn^{A_2^Tt} dt+ \expn^{A_1t}K_1 K_2^T \expn^{A_2^T t} A_2^T dt,
                   \end{align*}
and integrating gives the result.
\end{proof}

\noindent Setting $A_1=A_2=A$ and $K_1=K_2=B\mathcal Q_M^{\frac{1}{2}}$ in Proposition 
\ref{prodeinfachh}, we see that $\int_0^t\expn^{As}B\mathcal Q_MB^T 
\expn^{A^Ts}ds$ solves the differential equation (\ref{matrixdiffgl}). Since 
the solution to (\ref{matrixdiffgl}) is unique, we have 
\begin{align}\label{finitreachgram}
 P(t)=\int_0^t\expn^{As}B\mathcal Q_MB^T \expn^{A^Ts}ds,\;\;\;t\geq 0.         
  \end{align}
Consequently, $x^TP(t)x$ is an increasing function. If $\mathcal Q_M = I$, then we obtain the reachability Gramian of the deterministic setting (\ref{detsysant}), see \cite{antoulas}. This is also the case if $M$ is a standard Wiener process.

The finite reachability Gramian $P(t)$ provides information about the reachability of a state which we see from the following identity:
\begin{align}\label{supkleinpt}
 \max_{t\in[0, T]}\mathbb E \left\langle x(t, 0, M), x \right\rangle^2_{\mathbb 
R^n}=x^T P(T)x\quad \text{for} \quad x\in\mathbb R^n.
\end{align}
Consequently, we know that $\left\langle x(t, 0, M), x 
\right\rangle_{\mathbb R^n}=0$, $t\in [0, T]$, $\mathbb P$ a.s. if and only if 
$x\in \kernel P(T)$ meaning that $x(t, 0, M)$ is orthogonal to $\kernel 
P(T)$. Since $P(T)$ is symmetric positive semidefinite, we have $\left(\kernel 
P(T)\right)^\perp=\im P(T)$ and hence \begin{align}\label{impreachset}
      \mathbb P\left\{x(t, 0, M)\in \im P(T),\;\;\;t\in[0, T]\right\}=1,
\end{align}
We observe from (\ref{impreachset}) that all the states that are not in $\im 
P(T)$ are not reachable and thus they do not contribute to the system 
dynamics. As a first step to reduce the system dimension it is necessary to 
remove all the states that are not in $\im P(T)$. We will see in the next 
Proposition, that the finite reachability Gramians can be replaced by the infinite 
Gramian 
\begin{align}\label{infinitreachgram}
P=\int_0^\infty\expn^{As}B\mathcal Q_MB^T \expn^{A^Ts}ds        
\end{align}
since their images coincide. This (infinite) Gramian exists due to the asymptotic stability of $A$. It is easier to work with $P$ since it can be computed as the unique solution to 
\begin{align}\label{lyapeqreach}
           AP+P A^T=-B\mathcal Q_MB^T.                               
\end{align}
$P$ satisfies (\ref{lyapeqreach}) since $P(t)$ satisfies (\ref{matrixdiffgl}) and $\dot P(t)=\expn^{At}B\mathcal Q_MB^T 
\expn^{A^Tt}\rightarrow 0$ if $t\rightarrow \infty$ due the asymptotic 
stability of $A$. For the case $\mathcal Q_M=I$ this Gramian was discussed in 
\cite[Section 4.3]{antoulas} in the context of balancing for deterministic 
systems (\ref{detsysant}).
\begin{prop}\label{refsameimbla}
The images of the finite reachability Gramians $P(t)$, $t>0$, and the 
infinite reachability Gramian $P$ are the same, that is,
\[
\im P(t)=\im P\;\;\;\text{for all } t> 0.
\]
\end{prop}
\begin{proof}
Since $P$ and $P(t)$ are symmetric positive semidefinite, it is enough 
to show that their kernels are equal. Let $v\in \kernel 
P$. This implies $0\leq v^T P(t) v\leq v^T P v=0$,
since $t\mapsto v^T P(t) v$ is increasing. Hence $v\in \kernel P(t)$. On the other hand, if $v\in \kernel P(t)$, we have 
$
0=v^T P(t) v= \int_0^t v^T \expn^{As}B\mathcal Q_MB^T \expn^{A^Ts} v ds.
$
Consequently, $v^T \expn^{As}B\mathcal Q_MB^T \expn^{A^Ts} v=0$ for all $s\in 
[0, t]$.  Since the entries of $\expn^{As}B\mathcal Q_MB^T \expn^{A^Ts}$ are 
analytic functions, the scalar function $f(t):=v^T\expn^{As}B\mathcal 
Q_MB^T \expn^{A^Ts} v$ is analytic, such that $f\equiv 0$ on $[0, \infty)$. 
Thus, 
$
0=\int_0^\infty v^T\expn^{As}B\mathcal Q_MB^T \expn^{A^Ts}v ds=v^T P v,
$
and the result follows.
\end{proof}

\noindent Let us now assume that we already removed all the unreachable states 
from (\ref{stochstatenew}). So, (\ref{completereach}) holds which 
implies that $\im P=\mathbb R^n$. We choose an orthonormal basis of $\mathbb R^n$, consisting of eigenvectors $\{p_k\}_{k=1}^n$ of $P$, and the 
following representation holds: 
\begin{align}\label{eq:expansion}
 x(t, 0, M)=\sum_{k=1}^n  \left\langle x(t, 0, M), p_k \right\rangle_{\mathbb 
R^n} p_k.
\end{align}
We investigate how much the noise influences $x(t, 0, M)$ in the direction of $p_k$. If a state remains close to zero, it barely contributes to the system 
dynamics. Those states can be identified with the help of the positive 
eigenvalues $\{\lambda_k\}_{k=1}^n$ of $P$. Using (\ref{supkleinpt}) and the fact 
that $P(T)$ is increasing, we obtain \begin{align}\label{interlambad}
 \max_{t\in[0, T]}\mathbb E \left\langle x(t, 0, M), p_k\right\rangle^2_{\mathbb 
R^n}=p_k^T P(T) p_k\leq p_k^T P p_k=\lambda_k.\end{align}
Hence, if $\lambda_k$ is small, then the the corresponding 
coefficient $\left\langle x(t, 0, M), p_k\right\rangle_{\mathbb 
R^n}$ in (\ref{eq:expansion}) is small (in the $L^2(\Omega, \mathcal{F}, \mathbb{P})$ sense). This 
means that the noise hardly steers the state in the direction of $p_k$. 
Consequently, the states that are difficult to reach are contained in the space 
spanned by the eigenvectors corresponding to the small eigenvalues of $P$. 


We continue by reasoning why using the modified reachability 
Gramian $P$ is better than using the reachability Gramian 
$P_D=\int_0^\infty\expn^{As}BB^T \expn^{A^Ts}ds$ ($\mathcal Q_M=I$) of the 
deterministic system (\ref{detsysant}). 
\begin{prop}\label{propdworse}
The following properties hold for the (modified) reachability Gramians $P$ and $P_D$:
\begin{itemize}
\item[(a)]In general, we have  $\im P \subseteq \im P_D.$                         
\item[(b)] If $\mathcal Q_M>0$ (positive definite), then $\im P = \im P_D$.                         
\item[(c)] If $B^T\ker P\neq \left\{0\right\}$, then 
$\im P \subset \im P_D$.                         
 \end{itemize}
\end{prop}
\begin{proof}
  Let $v\in \kernel P_D$, then 
\begin{align*}
0=v^T P_Dv= \int_0^\infty v^T\expn^{As}BB^T \expn^{A^Ts}v ds=\int_0^\infty 
\left\|B^T \expn^{A^Ts}v\right\|^2_{\mathbb R^n} ds,
\end{align*}
which is equivalent to $B^T \expn^{A^Ts}v\equiv 0$ on $\mathbb R_+$  and  
implies $\mathcal Q_M^\frac{1}{2} B^T \expn^{A^Ts}v\equiv 0$ on $\mathbb R_+$. 
Equivalently, we have 
\begin{align*}
0=\int_0^\infty \left\|\mathcal Q_M^\frac{1}{2}B^T 
\expn^{A^Ts}v\right\|^2_{\mathbb R^n} ds=v^T Pv,
\end{align*}
and since $v\in\kernel P$ if and only if $0=v^T Pv$, we have $\kernel P_D\subseteq 
\kernel P$. Consequently, we obtain $\im P\subseteq \im P_D$ due to  
$\left(\kernel P\right)^\perp =\im P$ and 
$\left(\kernel P_D\right)^\perp=\im P_D$.

If $\mathcal Q_M>0$, then $\mathcal Q_M^\frac{1}{2} B^T 
\expn^{A^Ts}v\equiv 0$ on $\mathbb R_+$ implies $B^T \expn^{A^Ts}v\equiv 0$ on 
$\mathbb R_+$. In this case, all the above statements are equivalent. Therefore $\kernel P_D= \kernel P$ and hence $\im P=\im P_D$.

To prove (c), assume $v\in \ker P$. Pre- and postmultiplying (\ref{lyapeqreach}) with $v^T$ and $v$, respectively, yields 
\begin{align*}
0=v^TB\mathcal Q_MB^Tv= \left\|\mathcal Q_M^{\frac{1}{2}}B^T 
v\right\|^2_{\mathbb R^n}.
\end{align*}
This implies $\mathcal Q_M^{\frac{1}{2}}B^Tv=0$ but if $B^T\ker P\neq 
\left\{0\right\}$, then there is a $v\in \kernel P$ such that $B^Tv\neq 0$. We 
set $f(t):=B^T \expn^{A^T t}v$, $t\geq 0$, and observe that $f$ is an analytic 
function that is not constantly zero since $f(0)=B^Tv\neq 0$. Consequently, $f$ 
has only countably many zeros such 
that $\left\|B^T\expn^{A^Ts}v\right\|^2_{\mathbb R^n}$ is a purely positive 
function up to Lebesgue zero sets. Hence, \begin{align*}
0<\int_0^\infty \left\|B^T \expn^{A^Ts}v\right\|^2_{\mathbb R^n} ds=v^T P_Dv,
\end{align*}
such that $v\not \in \ker P_D$. Having $\ker P_D\subset \ker P$ implies (c).
\end{proof}

\noindent By (\ref{impreachset}) and Proposition \ref{propdworse} (a), we obtain 
\begin{align}\label{pdnaja}
      \mathbb P\left\{x(t, 0, M)\in \im P_D,\;\;\;t\in[0, T]\right\}=1.
\end{align}

One could now think of using $P_D$ instead of $P$ but from (\ref{pdnaja}) not all unreachable states can be identified especially if case (c) in Proposition \ref{propdworse} holds. Hence, if we were to use $P_D$, we would underestimate the set of unreachable states. Even if we assume that the system is already completely reachable (e.g. (\ref{completereach}) holds), inequality (\ref{interlambad}) cannot be obtained with $P_D$. This means that we cannot identify the difficult to reach states with the help of the eigenvalues of $P_D$.
In Section \ref{btandspa} we will see that $P$, rather than $P_D$, enters the error bound for the ROM. 

Finally we note that the reachability Gramian $P_D$ of system (\ref{detsysant}) does not depend on the input $u$. If a ``noisy control'' is used, this does not apply, since $P$ depends on $\mathcal{Q}_M$ and hence on the L\'evy process $M$. 

\paragraph{Observability} We conclude this section by introducing a  deterministic observability 
concept for the output equation \begin{align*}
y(t, x_0, 0)= C x(t, x_0, 0),\;\;\;t\geq 0.
\end{align*}
corresponding to (\ref{stochstatenew}) with $M\equiv 0$. We 
recall known facts from \cite[Subsection 4.2.2]{antoulas} to characterise the importance of certain initial states in the system dynamics since we 
are in a situation without noise. We assume to have an unknown initial 
state $x_0\in\mathbb R^n$ in the following observation problem and aim to reconstruct $x_0$ from the observation $y$ on the entire time interval $[0, 
\infty)$. 
\begin{defn}
An initial state $x_0$ is not observable if $y(\cdot, x_0, 
0)\equiv 0$ on $[0, 
\infty)$, i.e. it cannot be reconstructed by the observation. Otherwise, $x_0$ is 
called observable. A system a called completely observable if every initial state is observable.
\end{defn}

\noindent In order to determine the observability of a state, we consider the energy that is caused by the observations of $x_0$: 
\begin{align}\label{observenergy}
 \int_0^\infty \left\| y(t, x_0, 0)\right\|^2_{\mathbb R^p}dt=x_0^T 
\int_0^\infty \expn^{A^T t}C^T C \expn^{At} dt\;x_0=x_0^T Q x_0,
\end{align}
where we used that $x(t, x_0, 0)=\expn^{At}x_0$ and set $Q=\int_0^\infty \expn^{A^T t}C^T C \expn^{At} dt$. The observability Gramian $Q$ exists due to the asymptotic stability of $A$ and is the unique solution to 
\begin{align}\label{observlyap}
         A^TQ+Q A=-C^T C.                                                        
          \end{align}
The above relation is obtained by replacing $A$ and $B\mathcal Q_MB^T$ in 
(\ref{infinitreachgram}) and (\ref{lyapeqreach}) by $A^T$ and $C^T C$, respectively. 

From (\ref{observenergy}) we see that $x_0$ is unobservable if and only if 
$x_0\in\ker Q$. Hence, the system is completely observable if and only if 
$\ker Q=\left\{0\right\}$. Besides the unobservable states we aim to remove the 
difficult to observe states from the system in order to obtain an accurate 
ROM. The difficult to observe states are those producing only 
little observation energy, i.e. the corresponding observations $y$ are close to 
zero in the $L^2$ sense. Using (\ref{observenergy}) again, the difficult to 
observe states are contained in the eigenspaces spanned by the eigenvectors of 
$Q$ corresponding to the small eigenvalues.

\subsection{Balancing related MOR}\label{subsecprocedure}

Before considering balanced truncation (BT) and singular perturbation approximation (SPA) we summarise the general theory for balancing and how to find a balancing transformation. 

States that are difficult to reach have large components in the span of the 
eigenvectors corresponding to small eigenvalues of the reachability Gramian 
$P$, cf. (\ref{interlambad}). Similarly, states that are difficult to 
observe are the ones that have large components in the span of eigenvectors 
corresponding to small eigenvalues of the observability Gramian $Q$, see 
(\ref{observenergy}). Hence in order to produce accurate ROMs 
one eliminates states that are both difficult to reach and difficult to observe. 
To this end we need to find a basis in which the dominant reachable and 
observable states are the same, which is done by a simultaneous transformation 
of the Gramians. 

Let $T\in\mathbb{R}^{n\times n}$ be a nonsingular matrix. Transforming the states using 
\[\hat{x}(t) = Tx(t),\] 
the system (\ref{linsysafterdisintr2}) becomes
\begin{align}\label{balancingtransformation}
d\hat{x}(t)&=\hat{A} \hat{x}(t)dt+ \hat{B}dM(t),\\ \nonumber
y(t)&= \hat{C} \hat{x}(t),
\end{align}
where $\hat{A} = TAT^{-1}$, $\hat{B}=TB$, $\hat{C} =CT^{-1}$. The input-output 
map remains the same, only the state, input and output matrices are 
transformed.  

$P$ and $Q$, the reachability and observability Gramians of the original systems which satisfy (\ref{lyapeqreach}) and (\ref{observlyap}) can be transformed 
into reachability and observability Gramians of the transformed system $\hat{P} 
=TPT^T $ and $\hat{Q} = T^{-T}QT^{-1}$ (by multiplying (\ref{lyapeqreach}) with $T$ from the left and $T^T$ from the right and (\ref{observlyap}) with $T^{-T}$ from the left and $T^{-1}$ from the right). The Hankel singular values (HSVs) 
of $\sigma_1\ge\ldots\ge\sigma_n$, where $\sigma_i 
=\sqrt{\lambda_i(PQ)},\,i=1,\ldots,n$ for the original and transformed system 
are the same. The above transformation is a balancing transformation if the 
transformed Gramians are equal to each other and diagonal. Such a transformation 
always exists if $P, Q>0$ and can be obtained by choosing 
\[
T=\Sigma^{-\frac{1}{2}}U^T L^T \quad\text{and}\quad T^{-1}=KV\Sigma^{-\frac{1}{2}},
\]
where $\Sigma=\diag(\sigma_{1},\ldots,\sigma_n)$ are the HSVs. $Y$, $Z$, $L$ and $K$ are computed as follows. Let $P = KK^T$, $Q=LL^T$ 
be square root factorisations of $P$ and $Q$, then an SVD of $K^TL = V\Sigma 
U^T$ gives the required matrices. With this transformation 
$\hat{P}=\hat{Q}=\Sigma$.  

Below, let $T$ be the balancing transformation as stated above, then we 
partition the coefficients of the balanced realisation as 
follows:\begin{align}\label{balancedrels}
T{A}T^{-1}= \mat{cc}{A}_{11}&{A}_{12}\\ 
{A}_{21}&{A}_{22}\rix,\;\;\; T{B} = \mat{c}{B}_1\\ {B}_2\rix,\;\;\;  
{CT^{-1}} = \mat{cc}{C}_1 &{C}_2\rix, \end{align}
where ${A}_{11}\in\R^{r\times r}$ etc.                          
Furthermore, setting $\hat x= \mat{cc}{x}_1\\ 
{x}_2\rix$, where $x_1(t)\in\mathbb{R}^r$, we obtain the transformed partitioned system 
\begin{align}
\mat{cc}d{x}_1(t)\\d{x}_2(t)\rix &=\mat{cc}{A}_{11}&{A}_{12}\\ 
{A}_{21}&{A}_{22}\rix 
\mat{cc}{x}_1(t)\\ {x}_2(t)\rix dt+ \mat{c}{B}_1\\ {B}_2\rix dM(t),\label{balrelpart}\\ 
y(t)&= \mat{cc}{C}_1 &{C}_2\rix \mat{cc}{x}_1(t)\\ 
{x}_2(t)\rix.\label{balrelpartout}
\end{align}
In this system, the difficult to reach and observe states are 
represented by $x_2$, which correspond to the smallest HSVs  $\sigma_{r+1}, \ldots, \sigma_n$, but of course $r$ has  
to be chosen such that the neglected HSVs are small 
($\sigma_{r+1}\ll\sigma_{r}$). 

We discuss two methods (BT and SPA) to neglect $x_2$ leading to a reduced system of the form
\begin{align}\label{generalreducedsys}
dx_r(t)&=A_r x_r(t)dt+ B_rdM(t),\\ \nonumber
y_r(t)&= C_r x_r(t),
\end{align}
where $A_r\in\mathbb R^{r\times r}$, $B_r\in\mathbb R^{r\times m}$ and 
$C_r\in\mathbb R^{p\times r}$ ($r\ll n$). 

\paragraph{Balanced truncation} For BT the second row in (\ref{balrelpart}) is truncated and the remaining $x_2$ components in the first row and in (\ref{balrelpartout}) are set to zero. This leads to reduced coefficients 
\begin{align*}
 (A_r, B_r, C_r)=(A_{11}, B_1, C_1),
\end{align*}
which is similar to the deterministic case. The next lemma states that BT preserves asymptotic stability, which is known from the deterministic case, see \cite[Theorem 7.9]{antoulas}. 
\begin{lem}\label{btpresstab}
Let the Gramians $P$ and $Q$ be positive definite and $\sigma_r\neq \sigma_{r+1}$, 
then $\sigma\left(A_{ii}\right)\subset \mathbb C_-$ for $i=1, 2$, i.e. $A_{11}$ 
and $A_{22}$ are asymptotically stable. 
\end{lem}

\noindent The above lemma is vital for the error bound analysis in Section 
\ref{btandspa}.

\paragraph{Singular perturbation approximation} Instead of 
setting $x_2\equiv 0$, one assumes $\dot x_2\equiv 0$. This idea originates from the deterministic case, where it can be observed that $x_2$ are the fast 
variables meaning that they are in a steady state after a short time. In our 
framework, the classical derivative of $x_2$ does not exist but we proceed with 
setting $dx_2\equiv 0$ in (\ref{balrelpart}). This yields an algebraic constraint
\begin{align}\label{algbraicconst}
0=\int_0^t\mat{cc}{A}_{21}&{A}_{22}\rix \mat{cc}{x}_1(s)\\ {x}_2(s)\rix 
ds+{B}_2 M(t)=:R(t),
\end{align}
where we assumed zero initial conditions. Applying Ito's product formula (\ref{profriot}) 
to every summand of $R^T R=R_1^2+\ldots+R_{n-r}^2$ ($R_i$ is the $i$th 
component of $R$) yields
\begin{align*}
R^T(t) R(t)=\int_0^t d R^T(s) R(s)+\int_0^t R^T(s) dR(s)+\sum_{i=1}^{n-r}[R_i, 
R_i]_t.                                                                    
\end{align*}
Inserting the differential of $R$ and exploiting that the expectation of the 
Ito integrals is zero, gives 
\begin{align*}
\mathbb E\left[R^T(t) R(t)\right]=\mathbb E \left[\int_0^t a^T(s) 
R(s)ds\right]+\mathbb E\left[\int_0^t R^T(s) a(s) ds\right] +\mathbb E
\sum_{i=1}^{n-r}[R_i, R_i]_t,                                                   
                     \end{align*}
where we set $a(s)={A}_{21}{x}_1(s)+{A}_{22}{x}_2(s)$. Setting $R\equiv 0$ gives $0=\mathbb E\sum_{i=1}^{n-r}[R_i, R_i]_t$. Since $R_i$ and $e_i^T B_2 M$ have the same 
martingale parts and the same jumps, their compensator processes coincide by 
(\ref{decomqucov}) and hence  \begin{align*}
    0=\mathbb E\sum_{i=1}^{n-r}[R_i, R_i]_t=\mathbb E \sum_{i=1}^{n-r}[e_i^T B_2 
M, e_i^T B_2 M]_t.
    \end{align*} 
 Applying Ito's product formula to $(B_2 M)^T B_2M$ and taking the expectation, 
we have \begin{align*}
    0=\mathbb E \sum_{i=1}^{n-r}[e_i^T  B_2 M, e_i^T B_2 M]_t=\mathbb E (B_2 
M)^T B_2M
    \end{align*} 
which implies $B_2 M=0$ $\mathbb P$-a.s. Using this simplification in 
(\ref{algbraicconst}) yields \begin{align}\label{a22welldef}
 x_2(t)=-A_{22}^{-1} A_{21} x_1(t),
                             \end{align}
which is well-defined by Lemma \ref{btpresstab} and which we use in the first row of (\ref{balrelpart}) and in (\ref{balrelpartout}). This leads to reduced order coefficients 
\begin{align}\label{coefficientspa}
   (A_r, B_r, C_r)=(A_{11}- A_{12} A_{22}^{-1} A_{21}, B_1, C_1-C_2 A_{22}^{-1} 
A_{21}).    \end{align}
This reduced model is different to the deterministic case, that requires  
$B_r=B_1-A_{12} A_{22}^{-1}B_2$ with an additional term in the output equation which does not depend on the state, see \cite[Section 2]{spa}. In the deterministic case, the ROM is balanced \cite{spa}, which is not true here due to the modification. Like in the deterministic case, the observability Gramian is given by 
$Q_r=\Sigma_1=\diag(\sigma_1, \ldots, \sigma_r)$. This property is obtained by multiplying (\ref{partobeq}) with $\hat A^{-T}$ ($\hat A$ is the matrix of the balanced system) from the left and with $\hat A^{-1}$ from the right and then evaluating the left upper block of the resulting equation, see also 
(\ref{bzwzw}). The following example shows that the reduced order reachability is $P_r$ is not equal to $\Sigma_1$ in general which is different from the deterministic case. 
\begin{example}
Let $M$ be a standard Wiener process, then $\mathcal Q_M=I$ and set
\begin{align*}
A=\left(\begin{smallmatrix}
  -2&-\frac{4}{3}&-\frac{4}{5}\\
  -\frac{4}{3}&-1&-\frac{2}{3}\\
  -\frac{4}{5}&-\frac{2}{3}&-\frac{1}{2}
    \end{smallmatrix}\right)\;\;\;\text{and}\;\;\; 
C^T=B=\left(\begin{smallmatrix}
  4\\
  2\\
  1
    \end{smallmatrix}\right).
    \end{align*}             
$A$ is asymptotically stable and the system is balanced since 
$P=Q=\diag(4, 2, 1)$. We fix the reduced order dimension to $r=2$ and compute the reduced order coefficients by SPA in (\ref{coefficientspa}). We know that $Q_r=\diag(4, 2)$ but the reachability 
Gramian is up to the digits shown $P_r=\left(\begin{smallmatrix}
  10.1604&2.5668&\\
  2.5668&11.8396
    \end{smallmatrix}\right)$ which we computed numerically. This implies that 
the HSVs are not a subset of the original ones anymore. Here, 
they are $6.5822$ and $4.5822$.
   \end{example}
We conclude this Section by a stability result from \cite{spa}. 
\begin{lem}\label{spapresstab}
Let the Gramians $P$ and $Q$ be positive definite and $\sigma_r\neq \sigma_{r+1}$, 
then $\sigma\left(A_{ii}- A_{ij} A_{jj}^{-1} A_{ji}\right)\subset \mathbb C_-$ 
for $i, j=1, 2$ with $i\neq j$. 
\end{lem}

\noindent The inverses in Lemma \ref{spapresstab} exist because of the asymptotic 
stability of the matrices $A_{11}$ and $A_{22}$, see Lemma \ref{btpresstab}.

\subsection{Error bounds for BT and SPA}\label{btandspa}

Before we specify the error bounds for BT and SPA, we provide a general error bound comparing the 
outputs of (\ref{stochsysnew}) and (\ref{generalreducedsys}) with 
asymptotically stable matrices $A$, $A_r$ and initial conditions $x_0=0$, 
$x_{r, 0}=0$. These outputs are then given by 
Ornstein-Uhlenbeck processes \begin{align*}
 y(t)&= C x(t)= C \int_0^t \expn^{A(t-s)} B dM(s), \\
y_r(t)&= C_r x_r(t)= C_r \int_0^t \expn^{A_r(t-s)} B_r dM(s)
\end{align*}
as mentioned in \cite{applebaumendlich}, see also \cite{ornsteinrep1, 
ornsteinuhlenbecl2}. Using these representations and Cauchy's inequality as well as Ito's isometry (see \cite{zabczyk}), we obtain 
\begin{align*}
&\mathbb E \left\|y(t)- y_r(t)\right\|_{\mathbb R^p}\leq \left(\mathbb E 
\left\|y(t)- y_r(t)\right\|_{\mathbb R^p}^2\right)^{\frac{1}{2}} 
\\&= \left(\mathbb E 
\left\|\int_0^t \left(C\expn^{A(t-s)} B - C_r \expn^{A_r(t-s)} B_r\right) 
dM(s)\right\|_{\mathbb R^p}^2\right)^{\frac{1}{2}}\\&=\left(\int_0^t 
\left\| \left(C\expn^{A(t-s)} B - C_r \expn^{A_r(t-s)} B_r\right)\mathcal 
Q_M^{\frac{1}{2}}\right\|_{F}^2 ds\right)^{\frac{1}{2}},
\end{align*}
where $\mathcal Q_M$ is the covariance matrix of $M$. Substitution and taking 
 limits yields 
\begin{align*}
 \mathbb E \left\|y(t)- y_r(t)\right\|_{\mathbb R^p}\leq \left(\int_0^\infty 
\left\| \left(C\expn^{As} B - C_r \expn^{A_r s} B_r\right)\mathcal 
Q_M^{\frac{1}{2}}\right\|_{F}^2 ds\right)^{\frac{1}{2}}.
\end{align*}
Using the definition of the Frobenius norm and the linearity of the 
trace, we have \begin{align}\label{computeerrorbound}
 \sup_{t\in[0, T]}\mathbb E \left\|y(t)- y_r(t)\right\|_{\mathbb R^p}\leq 
\left(\trace\left(C P C^T\right)+\trace\left(C_r P_r 
C_r^T\right)-2\trace \left(C P_g C_r^T\right)\right)^{\frac{1}{2}},
\end{align}
where $P=\int_0^\infty\expn^{As}B\mathcal Q_MB^T \expn^{A^Ts}ds$ is the reachability Gramian of the original system satisfying (\ref{lyapeqreach}), $P_r=\int_0^\infty\expn^{A_r s}B_r\mathcal Q_MB_r^T \expn^{A_r^Ts}ds$ the one 
of the reduced system satisfying
\begin{align*}
A_{r} P_r+P_r A_{r}^T =-B_r \mathcal Q_M B_r^T                            
    \end{align*}
and $P_g=\int_0^\infty\expn^{A s}B\mathcal Q_MB_r^T \expn^{A_r^Ts}ds$ is 
the solution to \begin{align}\label{gemischtegram}
A P_g+P_g A_{r}^T =-B \mathcal Q_M B_r^T.                            
    \end{align}
Equation (\ref{gemischtegram}) is a consequence of Proposition 
\ref{prodeinfachh} with $A_1=A$ and $A_2=A_r$ and the fact that $\expn^{A s}B\mathcal Q_MB_r^T \expn^{A_r^Ts}\rightarrow 0$ if $s\rightarrow\infty$ due to the asymptotic 
stability of $A$ and $A_r$. The matrices $P, P_r$ and $P_g$ in the 
bound in (\ref{computeerrorbound}) are all well-defined because $A$ and $A_r$ 
are asymptotically stable. The error bound in (\ref{computeerrorbound}) holds for both BT and SPA since both approaches preserve asymptotic stability, see Lemmas \ref{btpresstab} and \ref{spapresstab}.

For both BT and SPA the 
representation in (\ref{computeerrorbound}) can be used for practical 
computations of the error bound. The Gramian $P$ is already available since it 
is required in the balancing procedure. The reduced model Gramian $P_r$ is 
computationally cheap because it is low dimensional assuming that we fix a 
small ROM dimension. The same is true for $P_g$ since it has only a 
few columns which makes the solution to (\ref{gemischtegram}) easily accessible. Since the error bound (\ref{computeerrorbound}) is computationally cheap, it can be computed for several ROM dimensions and hence be used to 
find a suitable $r$.

In the next two Theorems we specify the general error bound in (\ref{computeerrorbound}) for both BT and SPA and represent it in terms of the truncated HSVs $\sigma_{r+1}, \ldots, \sigma_n$ of the system. Using the balanced realisation (\ref{balancedrels}) of the original system with $\hat 
P=\hat Q=\Sigma$ and its corresponding partition, we have \begin{align} 
\label{partreacheq}
\mat{cc}{A}_{11}&{A}_{12}\\ 
{A}_{21}&{A}_{22}\rix \hspace{-0.1cm}\mat{cc}\Sigma_1\hspace{-0.2cm}{}&\\ 
&\Sigma_2\rix \hspace{-0.05cm}+\hspace{-0.05cm} 
\mat{cc}\Sigma_1\hspace{-0.2cm}{}&\\ 
&\Sigma_2\rix \hspace{-0.1cm} \mat{cc}A^T_{11}&A^T_{21}\\ 
A^T_{12}& A^T_{22}\rix\hspace{-0.05cm}=& \hspace{-0.05cm}
-\mat{cc}\tilde {B}_1\tilde B_1^T&\tilde {B}_1\tilde B_2^T \\ 
\tilde {B}_2\tilde B_1^T& \tilde {B}_2\tilde B_2^T\rix \\ \label{partobeq}
\mat{cc}{A}^T_{11}&{A}^T_{21}\\ 
{A}^T_{12}&{A}^T_{22}\rix \hspace{-0.1cm}\mat{cc}\Sigma_1\hspace{-0.2cm}{}&\\ 
&\Sigma_2\rix \hspace{-0.05cm}+\hspace{-0.05cm} 
\mat{cc}\Sigma_1\hspace{-0.2cm}{}&\\ 
&\Sigma_2\rix \hspace{-0.1cm} \mat{cc}A_{11}&A_{12}\\ 
A_{21}& A_{22}\rix\hspace{-0.05cm}=& \hspace{-0.05cm}
-\mat{cc}C_1^TC_1&C_1^TC_2 \\ 
C^T_2C_1& C^T_2C_2\rix
\end{align}
where $\mat{c}\tilde {B}_1\\ \tilde B_2 
\rix=\mat{c}{B}_1\mathcal{Q}_M^{\frac{1}{2}}\\ B_2 
\mathcal{Q}_M^{\frac{1}{2}}
\rix$, $\Sigma_1=\diag(\sigma_{1}, \ldots, \sigma_r)$ and 
$\Sigma_2=\diag(\sigma_{r+1}, \ldots, \sigma_n)$.

\begin{thm}\label{bterrorbound}
Let $y_{BT}$ be the output of the reduced order system obtained by BT, then under the assumptions of Lemma \ref{btpresstab}, we 
have\begin{align*}
 \sup_{t\in[0, T]}\mathbb E \left\|y(t)- y_{BT}(t)\right\|_{\mathbb 
R^p}\leq \left(\trace(\Sigma_2 ( B_2 \mathcal{Q}_M B_2^T+2 P_{g, 2} 
A_{21}^T))\right)^{\frac{1}{2}},\end{align*}
where $P_{g, 2}$ are the last $n-r$ rows of $T P_g$ with $T$ being the 
balancing transformation. 
\end{thm}
\begin{proof}
Evaluating the left and right upper block of 
(\ref{partobeq}) yields\begin{align} \label{firstc1c2bt2jk}
 A_{11}^T \Sigma_1+\Sigma_1 A_{11}&=-C_1^T C_1\\
 \label{c1c2eqbt2} A_{21}^T \Sigma_2+\Sigma_1 A_{12} &=-C_1^T 
C_2.
\end{align}
From (\ref{computeerrorbound}) the error bound has the form
\begin{align}
\label{eq:epserror}
\epsilon=\sqrt{\trace(C PC^T)+\trace(C_1 P_r C_1^T)-2\trace(C P_g C_1^T)},
\end{align}
since $C_r=C_1$. Using the balancing transformation $T$ and the partition of 
$C T^{-1}$ in (\ref{balancedrels}), we obtain $\trace(C PC^T)=\trace(C T^{-1} T 
P T^T (C T^{-1})^T)=\trace(C T^{-1} \Sigma (C T^{-1})^T)=\trace(C_1\Sigma_1 
C_1^T)+\trace(C_2 \Sigma_2 C_2^T)$. Now, the left upper block of 
(\ref{partreacheq}) is \begin{align} \label{blaaa}
 A_{11} \Sigma_1+\Sigma_1 A_{11}^T=-B_1 \mathcal{Q}_M B_1^T
 \end{align}
 such that $P_r=\Sigma_1$. Using the 
partitions of $CT^{-1}$ and $T P_g=\mat{c}
P_{g, 1} \\
P_{g, 2}
\rix$, we obtain $\trace(C P_g C_1^T)=\trace(C T^{-1} T P_g 
C_1^T)=\trace(C_1 P_{g, 1} C_1^T)+\trace(C_2 P_{g, 
2} C_1^T)$. Inserting these results into (\ref{eq:epserror}) gives
\begin{align}\label{insertforebbt2fh}
\epsilon^2=2\trace(C_1 \Sigma_1 C_1^T)+\trace(C_2 \Sigma_2 
C_2^T)-2\trace(C_1 P_{g, 1} C_1^T)-2\trace(C_2 P_{g, 2} C_1^T).
          \end{align}
Using $\trace(C_2 P_{g, 2} C_1^T)=\trace(P_{g, 2} C_1^T C_2 )$ and substituting (\ref{c1c2eqbt2}) yields 
\begin{align*}
\trace(C_2 P_{g, 2} C_1^T)=-\trace(P_{g,2} (A_{21}^T \Sigma_2+\Sigma_1 A_{12}))=-\trace(\Sigma_2 P_{g, 2} 
A_{21}^T)-\trace(\Sigma_1 A_{12} P_{g, 2}).
          \end{align*}
Multiplying (\ref{gemischtegram}) with the balancing transformation 
from the left and using the partitions of $TAT^{-1}$ and $TB$ from 
(\ref{balancedrels}) yields 
\begin{align*}
\mat{cc}{A}_{11}&{A}_{12}\\ 
{A}_{21}&{A}_{22}\rix \mat{c}
P_{g, 1} \\
P_{g, 2}
\rix+\mat{c}
P_{g, 1} \\
P_{g, 2}
\rix A_{11}^T =-\mat{c}{B}_1\\ {B}_2\rix \mathcal Q_M B_1^T.                    
    \end{align*} 
With the first row of this equation, $A_{11}P_{g, 1}+P_{g, 1} A_{11}^T+B_1\mathcal Q_M B_1^T=-A_{12} P_{g, 2},$
we have\begin{align*}
-\trace(C_2 P_{g, 2} C_1^T)=-\trace(\Sigma_1 (B_1\mathcal{Q}_M B_1^T+A_{11} 
P_{g, 1}+P_{g, 1} A_{11}^T))+\trace(\Sigma_2 P_{g, 2} A_{21}^T),
          \end{align*}
and substituting (\ref{firstc1c2bt2jk}), we obtain  
\begin{align*}
\trace(\Sigma_1 (A_{11}P_{g, 1}+P_{g, 1} A_{11}^T))=\trace(P_{g, 
1}(\Sigma_1 A_{11}+A_{11}^T \Sigma_1))=-\trace(P_{g, 1} C_1^T C_1),
\end{align*}
so that $-\trace(C_2 P_{g, 2} C_1^T)=\trace(\Sigma_2 P_{g, 2} A_{21}^T)-\trace( 
 \Sigma_1 B_1\mathcal Q_M B_1^T )+\trace(C_1 P_{g, 1} C_1^T).$
Inserting this result into (\ref{insertforebbt2fh}) 
gives\begin{align*}
 \epsilon^2=\trace(\Sigma_2 (C_2^T C_2+2 P_{g, 2} A_{21}^T 
))+2\trace(\Sigma_1 C_1^TC_1)-2\trace(\Sigma_1 B_1\mathcal{Q}_M B_1^T ).
          \end{align*}
With (\ref{firstc1c2bt2jk}) and (\ref{blaaa}), and the properties of the trace function we obtain  \begin{align*}
-\trace(\Sigma_1 B_1\mathcal{Q}_M B_1^T)=\trace(\Sigma_1 (A_{11} 
\Sigma_1+\Sigma_1 A_{11}^T))=
-\trace(\Sigma_1 C_1^T C_1).
          \end{align*}
Similarly $\trace(\Sigma_2 C_2^T C_2))=\trace(\Sigma_2 B_2 \mathcal{Q}_M 
B_2^T))$ can be shown using the right lower blocks of (\ref{partreacheq}) and 
(\ref{partobeq}). Hence,
\begin{align*}
 \epsilon^2=\trace(\Sigma_2 (B_2 \mathcal{Q}_M B_2^T+2 P_{g, 2} A_{21}^T )),
          \end{align*}
which gives the result.
\end{proof}

\begin{thm}\label{spaerrorbound}
Let $y_{SPA}$ be the output of the reduced order system obtained by SPA, then under the assumptions of Lemma 
\ref{spapresstab}, we 
have\begin{align*}
 \sup_{t\in[0, T]}\mathbb E \left\|y(t)- y_{SPA}(t)\right\|_{\mathbb 
R^p}\leq \left(\trace(\Sigma_2 (B_2 \mathcal{Q}_M B_2^T-2(A_{22}P_{g, 2}+A_{21} 
P_{g, 1})(A_{22}^{-1} A_{21})^T))\right)^{\frac{1}{2}},\end{align*}
where $P_{g, 1}$ are the first $r$ and $P_{g, 2}$ the last $n-r$ rows of $T 
P_g$ 
with $T$ being the balancing transformation. 
\end{thm}
\begin{proof}
Let $TAT^{-1}=\hat A=\mat{cc}{A}_{11}&{A}_{12}\\ 
{A}_{21}&{A}_{22}\rix$, then, since $A_{11}, A_{22}$ are invertible by Lemma 
\ref{btpresstab}, its inverse is given in block form \begin{align}
\hat A^{-1} \mat{cc}
\bar A^{-1}& -A_{11}^{-1} A_{12}(A_{22}-A_{21}A_{11}^{-1}A_{12})^{-1}\\
-A_{22}^{-1}A_{21}\bar A^{-1}& (A_{22}-A_{21}A_{11}^{-1}A_{12})^{-1}
\rix,
\end{align}
where $\bar A=A_{11}-A_{12}A_{22}^{-1}A_{21}$. If we multiply (\ref{partobeq}) 
with $\hat A^{-T}$ from the left hand side and select the left and right upper 
block of this equation, we obtain
\begin{align*}
 \Sigma_1+\bar A^{-T}(\Sigma_1 A_{11}-A_{21}^T A_{22}^{-T}\Sigma_2 
A_{21})&=-\bar A^{-T}\bar C^T C_1,\\
\bar A^{-T}(\Sigma_1 A_{12}-A_{21}^T A_{22}^{-T}\Sigma_2 A_{22})&=-\bar 
A^{-T}\bar C^T C_2,
\end{align*}
where $\bar C=C_{1}-C_{2}A_{22}^{-1}A_{21}$ and thus
\begin{align}\label{barcceins}
 \bar A^{T}\Sigma_1+\Sigma_1 A_{11}-A_{21}^T A_{22}^{-T}\Sigma_2 A_{21}&=-\bar 
C^T C_1,\\ \label{barcczwei} \Sigma_1 A_{12}-A_{21}^T A_{22}^{-T}\Sigma_2 
A_{22}&=-\bar C^T C_2.
\end{align}
Furthermore, multiplying (\ref{partobeq}) 
with $\hat A^{-T}$ from the left and with $\hat A^{-1}$ from the 
right, the resulting left upper block of the equation is 
\begin{align*}
\bar A^{-T} \Sigma_1+\Sigma_1 \bar A^{-1}=-\bar A^{-T}\bar C^T \bar C \bar 
A^{-1}
\end{align*}
and thus\begin{align}\label{bzwzw}
\bar A^T \Sigma_1+\Sigma_1 \bar A=-\bar C^T \bar C.
\end{align}
We define $\epsilon:=\left(\trace\left(C P C^T\right)+\trace\left(\bar C P_r 
\bar C^T\right)-2\;\trace\left(C P_g \bar C^T\right)\right)^{\frac{1}{2}}$ 
which is the error bound for SPA. From the 
proof of Theorem \ref{bterrorbound} we know that the following 
holds\begin{align*}
\trace\left(C P C^T\right)=\trace\left(C_1 \Sigma_1 C_1^T\right)+\trace\left(C_2 
\Sigma_2 C_2^T\right)=\trace\left( \Sigma_1 B_1\mathcal Q_M 
B_1^T\right)+\trace\left( \Sigma_2 B_2\mathcal Q_M B_2^T\right).
                                            \end{align*}
By (\ref{bzwzw}) and the definition of the reachability equation of 
the ROM, we have \begin{align*}
\trace\left(\bar C P_r \bar C^T\right)&=\trace\left(P_r \bar 
C^T\bar C \right)=-\trace\left(P_r (\bar A^T 
\Sigma_1+\Sigma_1 \bar A)\right)= -\trace\left(\Sigma_1 (P_r\bar A^T 
+\bar A P_r)\right)\\&=\trace\left(\Sigma_1 B_1\mathcal{Q}_MB_1^T\right).
\end{align*}
This leads to
\begin{align*}\epsilon^2=2\trace\left( \Sigma_1 B_1\mathcal Q_M 
B_1^T\right)+\trace\left( \Sigma_2 B_2\mathcal Q_M B_2^T\right)-2\;\trace\left(C 
P_g \bar C^T\right).
 \end{align*}
We multiply (\ref{gemischtegram}) with the balancing transformation 
$T$ from the left (here $A_r=\bar A$) and use the partitions of $TAT^{-1}$,  
$TB$ from (\ref{balancedrels}) and the partition of $T P_g=\mat{c}
P_{g, 1} \\
P_{g, 2}
\rix$. Thus,
\begin{align}\label{partmixedeqspa}
\mat{cc}{A}_{11}&{A}_{12}\\ 
{A}_{21}&{A}_{22}\rix \mat{c}
P_{g, 1} \\
P_{g, 2}
\rix+\mat{c}
P_{g, 1} \\
P_{g, 2}
\rix \bar A^T =-\mat{c}{B}_1\\ {B}_2\rix \mathcal Q_M B_1^T.                    
    \end{align}
We obtain $\trace\left(C P_g \bar C^T\right)=\trace\left(CT^{-1}T P_g \bar 
C^T\right)=\trace\left(C_1 P_{g, 1} \bar C^T\right)+\trace\left(C_2 P_{g, 2} 
\bar C^T\right)$ using the partition of $CT^{-1}$ in (\ref{balancedrels}). With  (\ref{barcczwei}) we obtain \begin{align*}
\trace(C_2 P_{g, 2} \bar C^T)&=-\trace(P_{g, 2}(\Sigma_1 A_{12}-A_{21}^T 
A_{22}^{-T}\Sigma_2 A_{22}))\\&=-\trace(\Sigma_1 A_{12}P_{g, 
2})+\trace(\Sigma_2 A_{22}P_{g, 2} A_{21}^T A_{22}^{-T})).
\end{align*}
Inserting the upper block of (\ref{partmixedeqspa}) leads to 
\begin{align*}
\trace(C_2 P_{g, 2} \bar C^T)=\trace(\Sigma_2 
A_{22}P_{g, 2} A_{21}^T A_{22}^{-T})+\trace(\Sigma_1(B_1\mathcal Q_M 
B_1^T+P_{g, 1} \bar A^T+ A_{11}P_{g,1})).
\end{align*}
Using (\ref{barcceins}) and the properties of the trace function we have 
\begin{align*}
\trace(\Sigma_1(P_{g, 1} \bar A^T+ A_{11}P_{g, 1}))
=-\trace(P_{g, 1}\bar C^T C_1-P_{g,1}(A_{22}^{-1}A_{21})^T \Sigma_2 A_{21}).
\end{align*}
Consequently, \begin{align*}
\trace\left(C P_g \bar C^T\right)=\trace(\Sigma_2 
A_{22}P_{g, 2} A_{21}^T A_{22}^{-T})+\trace(\Sigma_1B_1\mathcal Q_M 
B_1^T)-\trace(\Sigma_2 A_{21} P_{g, 1}(A_{22}^{-1}A_{21})^T)
\end{align*}
holds and hence \begin{align*}
\epsilon^2=\trace\left( \Sigma_2 B_2\mathcal Q_M 
B_2^T\right)-2\;(\trace(\Sigma_2 A_{22}P_{g, 2} 
(A_{22}^{-1}A_{21})^T)-\trace(\Sigma_2 A_{21} P_{g,1}(A_{22}^{-1}A_{21})^T)),
          \end{align*}
which provides the required result.          
\end{proof}

The error bound representations in Theorems \ref{bterrorbound} and 
\ref{spaerrorbound} depend on the $n-r$ smallest HSVs
$\sigma_{r+1}, \ldots, \sigma_n$. If the 
corresponding truncated components are unimportant, i.e. they are difficult to 
reach and observe, then the values 
$\sigma_{r+1}, \ldots, \sigma_n$ are small and consequently the 
error bound is small. Hence, the ROM is of good 
quality.

The error bounds in Theorems \ref{bterrorbound} and \ref{spaerrorbound} can be used to find 
a suitable reduced order dimension $r$. Small HSVs $\sigma_{r+1}, \ldots, \sigma_n$ for fixed $r$ would guarantee a small error.

Note that, if $\mathcal Q_M=I$, as for example in the standard Wiener case, then the 
error bound in Theorem \ref{bterrorbound} coincides with the 
$\mathcal{H}_2$-error bound in the deterministic case when using a normalised control, see \cite[Lemma 7.13]{antoulas}. 


The next section provides a particular SDE to which we will apply the theory developed in this section.

\section{Wave equations controlled by L\'evy noise}
\label{sec:wave}

In this section, we deal with a setting that covers the SPDE with its output in (\ref{inrobspeq})-(\ref{introbspoutnew}), a damped wave equation with additive noise which
can formally be interpreted as 
\begin{align}\label{nichtlinsde}
\ddot {\mathcal Z}(t)+\alpha \dot {\mathcal Z}(t)+\tilde {\mathcal A} {\mathcal 
Z}(t)+\tilde {\mathcal B} u(t)+\tilde {\mathcal D}_1 {\mathcal Z}(t-) \dot 
{\tilde M}_1(t)+\tilde {\mathcal D}_2 \dot {\mathcal Z}(t-)\dot {\tilde 
M}_2(t)=0, \end{align}
with $\tilde {\mathcal D}_i=0$ ($i=1, 2$), $\alpha>0$ and the $k$th 
component of the control $u_k\equiv \dot M_k$, $k\in \left\{1, \ldots, 
m\right\}$. Here, $M_1, \ldots, M_m$ are the components of an $\mathbb 
R^m$-valued L\'evy processes $M$ that is square integrable and has mean zero. 

This is in contrast to the setting in \cite{redbensec} where (\ref{nichtlinsde}) with multiplicative L\'evy noise was considered,
e.g. $\mathcal{D}_i\neq 0$ linear bounded operators and $u$ an $m$-dimensional stochastic control, $\tilde M_1$ and $\tilde M_2$ uncorrelated scalar L\'evy 
processes. For the stability analysis of the uncontrolled equation (\ref{nichtlinsde}) with 
Wiener noise ($u\equiv 0$) we refer to \cite{secord}.

Since L\'evy noise is no feasible control in the framework in \cite{redbensec}, this setting 
requires further analysis. We transform  
damped wave equation with additive noise into a first order SPDE and define the 
corresponding solution in Section \ref{setting}, following the approach in \cite{secord,redbensec}. In Section \ref{subsecspeltralgal}, we explain how the resulting first order SPDE can be approximated by a spectral Galerkin scheme. We refer to \cite{galerkin, galerkinhaus, galerkinjentzen, redbensec}, where similar techniques were applied.

\subsection{Setting and transformation into a first order SPDE}\label{setting}
Let $M=(M_1, \ldots, M_m)^T$ be square integrable L\'evy 
processes with zero mean that takes values in $\mathbb R^m$. Moreover, $M$ is 
defined on a complete probability space $\left(\Omega, \mathcal F, (\mathcal 
F_t)_{t\geq 0}, \mathbb P\right)$,\footnote{We assume 
that $\left(\mathcal F_t\right)_{t\geq 0}$ is right-continuous and that 
$\mathcal F_0$ contains all $\mathbb P$ null sets.} it is adapted to the 
filtration $(\mathcal F_t)_{t\geq 0}$ and its increments $M(t+h)-M(t)$ 
are independent of $\mathcal F_t$ for $t, h\geq 0$.

Let $\tilde {\mathcal 
A}:D(\tilde {\mathcal A})\rightarrow \tilde H$ be a 
self adjoint and positive definite operator on a separable Hilbert space $\tilde H$ and let $\{\tilde h_k\}_{k\in\mathbb N}$ be an orthonormal basis of eigenvectors of $\tilde {\mathcal A}$ for $\tilde H$, 
\begin{align}\label{eigvaleq}
\tilde {\mathcal A} \tilde h_k=\tilde \lambda_k \tilde h_k,
\end{align}
where $0<\tilde \lambda_1\leq \tilde \lambda_2\leq\ldots$ are the corresponding 
eigenvalues. We denote the well-defined square root of $\tilde {\mathcal A}$ 
by $\tilde {\mathcal A}^{\frac{1}{2}}$. $D(\tilde {\mathcal A}^{\frac{1}{2}})$ 
equipped with the inner product $\left\langle x, y \right\rangle_{D(\tilde 
{\mathcal A}^{\frac{1}{2}})}=\left\langle \tilde {\mathcal A}^{\frac{1}{2}} x, 
\tilde {\mathcal A}^{\frac{1}{2}} y\right\rangle_{\tilde H}$ represents a 
separable Hilbert space as well.

The (symbolic) second order SPDE we consider is given by 
\begin{align}\label{nichtlinsde2}
\ddot {\mathcal Z}(t)+\alpha \dot {\mathcal Z}(t)+\tilde {\mathcal A} {\mathcal 
Z}(t)+\sum_{k=1}^m\tilde {\mathcal B}_k \dot M_k(t)=0 
\end{align} 
with initial conditions 
${\mathcal Z}(0)=z_0$, $\dot {\mathcal Z}(0)=z_1$, $\alpha>0$ and output
\begin{align}\label{eq:output}
 \mathcal Y(t)= \mathcal C\left( \begin{smallmatrix}
  {\mathcal Z}(t) \\
  \dot {\mathcal Z}(t) 
 \end{smallmatrix}\right),\;\;\;t\geq 0.
\end{align}
We assume $\tilde {\mathcal B}_k\in\tilde H$
 and $\mathcal C\in L(D(\tilde {\mathcal A}^{\frac{1}{2}})\times \tilde H, 
\mathbb R^p)$. Using the separable Hilbert space $H=D(\tilde {\mathcal 
A}^{\frac{1}{2}})\times 
\tilde H$ with the inner product \begin{align*}
 \left\langle \left(\begin{smallmatrix}
  \tilde Z_1 \\
  \tilde Z_2 
 \end{smallmatrix}\right), \left( \begin{smallmatrix}
  \bar Z_1 \\
  \bar Z_2 
 \end{smallmatrix}\right)\right\rangle_{H}=\left\langle  
\tilde {\mathcal A}^{\frac{1}{2}} \tilde Z_1, \tilde {\mathcal A}^{\frac{1}{2}} 
\bar Z_1 
\right\rangle_{\tilde H}+\left\langle \tilde Z_2, \bar Z_2 
\right\rangle_{\tilde H}, \end{align*}
we transform this second order system into 
a first order system following the approach in \cite{secord, redbensec}. The 
system (\ref{nichtlinsde2})-(\ref{eq:output}) can be expressed as: 
\begin{align}\label{firstordertransf}
d\mathcal X(t)&=\mathcal A\mathcal X(t)dt+\sum_{k=1}^m\mathcal B_k 
dM_k(t),\;\;\;\mathcal 
X(0)=\mathcal{X}_0=\left( \begin{matrix}
  z_0 \\
  z_1 
 \end{matrix}\right),\\ 
\mathcal Y(t)&= \mathcal C\mathcal X(t),\;\;\;t\geq 0,\label{firstordertransfout}\end{align}
where \begin{align*}
\mathcal X(t)=\left( \begin{matrix}
  \mathcal Z(t) \\
  \dot {\mathcal Z}(t) 
 \end{matrix}\right)\in H,\;\;\;\mathcal A=\left[\begin{matrix}
                                   0 & I\\
                                   -\tilde {\mathcal A} & -\alpha I             
    \end{matrix}\right]\;\;\;\text{and}\;\;\;\mathcal B_k=\left[\begin{matrix}
                                   0 \\
                                   -\tilde {\mathcal B}_k                       
       \end{matrix}\right]\in H.\end{align*}
So far, we only worked with symbolic equation, since the classical 
derivative of a L\'evy process does not exist in general. The next lemma from 
\cite{pritzab} provides a stability result and is vital to define a c\`adl\`ag 
mild solution of (\ref{firstordertransf}).
\begin{lem}\label{lemwaveconexsemigr}
 For every $\alpha>0$ the linear operator $\mathcal A$ with domain $D(\tilde 
{\mathcal A})\times 
D(\tilde {\mathcal A}^{\frac{1}{2}})$ generates an exponential stable 
contraction 
semigroup $\left(S(t)\right)_{t\geq 0}$ with \begin{align*}
\left\| S(t)\right\|_{L(H)}\leq \expn^{-c t},\quad\text{where}\quad
       c\geq \frac{2 \alpha \tilde\lambda_1}{4 
\tilde\lambda_1+\alpha(\alpha+\sqrt{\alpha^2+4\tilde\lambda_1})}.
      \end{align*}
\end{lem}
We use this result to define the solution to (\ref{firstordertransf}).
\begin{defn}
An $(\mathcal F_t)_{t\geq 0}$-adapted c\`adl\`ag process $\left(\mathcal 
X(t)\right)_{t\geq 0}$ is called mild solution to (\ref{firstordertransf}) 
if for all $t\geq 0$ \begin{align}\label{mildsolalternative}
\mathcal X(t)=S(t) \mathcal{X}_0+\sum_{k=1}^m\int_0^t S(t-s) \mathcal B_k dM_k(s).
\end{align}
\end{defn}
We refer to \cite{zabczyk} for the definition of the Ito integral in 
(\ref{mildsolalternative}). For the finite dimensional case, the definition of 
an Ito integral with respect to L\'evy processes can be found in 
\cite{applebaumendlich} and Ito integrals with respect to martingales are 
defined in \cite{kuo}.

If we set $\tilde H=L^2([0, \pi])$, $D(\tilde {\mathcal 
A}^{\frac{1}{2}})=H_0^1([0, \pi])$, $\tilde {\mathcal A}=-\frac{\partial^2}{\partial \zeta^2}$, $\tilde {\mathcal B}_k=-f_k\in L^2([0, \pi])$ and the output operator $\mathcal C=\begin{bmatrix} \hat {\mathcal C} & 0\\ 
0& \hat {\mathcal C}
\end{bmatrix}$ with $\hat 
{\mathcal C}x=\frac{1}{2\epsilon}\int_{\frac{\pi}{2}-\epsilon}^{\frac{\pi}{2}
+\epsilon } x(\zeta) d\zeta$ $(x\in L^2([0, \pi])$ in system (\ref{nichtlinsde2})-(\ref{eq:output}) then we obtain the system in (\ref{inrobspeq})-(\ref{introbspoutnew}), which is therefore 
well-defined in the mild sense (\ref{mildsolalternative}).

\subsection{Numerical approximation}\label{subsecspeltralgal}

We study a spectral Galerkin scheme to approximate the mild 
solution of (\ref{firstordertransf}) with output  (\ref{firstordertransfout}), similar to the approach in \cite{galerkin,galerkinhaus,galerkinjentzen,redbensec} (mainly for 
SPDEs with Wiener noise). This approximation is based on a particular choice of 
an orthonormal basis $\{h_k\}_{k\in\mathbb N}$ of $H$, given by 
  \begin{align}\label{eigwavecae}
h_{2i-1}=\tilde\lambda_i^{-\frac{1}{2}} \begin{bmatrix}
   \tilde h_i \\
  0 
 \end{bmatrix}\quad\text{and}\quad h_{2i}=\begin{bmatrix}
  0 \\
  \tilde h_i 
 \end{bmatrix}\;\text{for}\; i\in\mathbb N,
 \end{align}
where $\{\tilde h_k\}_{k\in\mathbb N}$ and 
$\{\tilde \lambda_k\}_{k\in\mathbb N}$ are defined in (\ref{eigvaleq}), see \cite{redbensec}. In (\ref{inrobspeq})-(\ref{introbspoutnew}), we have $\tilde{\mathcal A}=-\frac{\partial^2}{\partial \zeta^2}$ on $[0, \pi]$. In this case $\tilde 
h_k=\sqrt{\frac{2}{\pi}} \sin(k\cdot)$ and $\tilde\lambda_k=k^2$ for $k\in\mathbb N$.

To approximate the $H$-valued process $\mathcal X$ in 
(\ref{firstordertransf}), we construct a sequence 
$\left(X_n\right)_{n\in\mathbb N}$ of finite dimensional adapted c\`adl\`ag 
processes with values in $H_n=\spaned\left\{h_1, \ldots, h_n\right\}$, 
defined by \begin{align}\label{approxsde}
d X_n(t)&= \mathcal{A}_n X_n(t)dt+\sum_{k=1}^m\mathcal B_{k,n} dM_k(t),\;\;\; X_n(0)=X_{0, 
n},\\ \nonumber
y_n(t)&=\mathcal{C}X_n(t) \;\;\;t\geq 0,
 \end{align}
where we set
\begin{itemize}
\item{$\mathcal A_n x=\sum_{k=1}^n \left\langle \mathcal A x, 
h_k\right\rangle_H h_k \in H_n$ for all $x\in D(\mathcal A)$,
}
\item{$\mathcal B_{k, n}=\sum_{k=1}^n \left\langle \mathcal B_k, 
h_k\right\rangle_H h_k \in H_n$ for $k\in \left\{1, \ldots, m\right\}$,
}
\item{$X_{0, n}=\sum_{k=1}^n \left\langle X_0, h_k\right\rangle_H h_k\in H_n$.
}
\end{itemize}
For the mild solution to (\ref{approxsde}), let $(S_n(t))_{t\geq 0}$ be a 
$C_0$-semigroup on $H_n$ given by
\begin{align*}
S_n(t) x=\sum_{k=1}^n \left\langle S(t) x, 
h_k\right\rangle_{H} h_k 
\end{align*}
for all $x\in H$. It is generated by $\mathcal A_n$ such that the mild solution 
of equation (\ref{approxsde}) is \begin{align*}
X_n(t)=S_n(t) X_{0, n}&+\int_0^t S_n(t-s) \mathcal B_{k, n} dM_k(s).
\end{align*}
Since $\mathcal A_n$ is bounded, the $C_0$-semigroup on 
$H_n$ is represented by $S_n(t)=\expn^{\mathcal A_n t}$, $t\geq 0$.
We formulate the main result of this section, which uses ideas from   
\cite{galerkin, galerkinhaus, galerkinjentzen, redbensec} and is proved in Appendix \ref{sec:app2}.
\begin{thm}\label{th:mildsol}
The mild solution $X_n$ of equation (\ref{approxsde}) approximates the 
mild solution $\mathcal X$ of equation (\ref{firstordertransf}), i.e. 
\begin{align*}
\mathbb E\left\|X_n(t)-\mathcal X(t)\right\|_{H}^2\rightarrow 0
\end{align*}
for $n\rightarrow \infty$ and $t\geq 0$. This implies the 
convergence of the corresponding outputs $\mathbb E\left\|y_n(t)-\mathcal 
Y(t)\right\|_{\mathbb{R}^p}^2\rightarrow 0$.
\end{thm}

\noindent In the following, we make use of the property that the mild and the strong solution of (\ref{approxsde}) coincide, since we are in finite dimensions. 

We write the output $y_n$ of the Galerkin system as an expression 
depending on the Fourier coefficients of the Galerkin solution $X_n$. The 
coefficients of $y_n$ are 
\begin{align*}
 y_n^\ell(t)=\left\langle y_n(t), e_\ell\right\rangle_{\mathbb 
R^p}&=\left\langle \mathcal C X_n(t), e_\ell\right\rangle_{\mathbb 
R^p}=\sum_{k=1}^n \left\langle \mathcal C h_k, 
e_\ell\right\rangle_{\mathbb R^p} \left\langle 
X_n(t), h_k\right\rangle_{H}
\end{align*}
for $\ell=1, \ldots, p$, where $e_\ell$ is the $\ell$-th unit vector in 
$\mathbb R^p$. We set 
\begin{align*}
x(t)=\left(\left\langle X_n(t), h_1\right\rangle_{H}, \ldots, 
\left\langle X_n(t), h_n\right\rangle_{H}\right)^T\; \text{and}\; 
C=\left(\left\langle \mathcal C h_k, e_\ell\right\rangle_{\mathbb 
R^p}\right)_{{\ell=1, \ldots, p} \atop {k=1, \ldots, n}}
\end{align*}
and obtain $y_n(t)=C x(t)$.
The components $x^i(t):=\left\langle X_n(t), h_i\right\rangle_{H}$ of $x(t)$ 
satisfy 
\begin{align*}
d x^i(t)= \left\langle \mathcal{A}_n X_n(t), h_i\right\rangle_{H} dt+ \sum_{k=1}^m
\left\langle \mathcal B_{k,n}, h_i\right\rangle_{H} dM_k(t).
\end{align*}
Using the Fourier series representation of $X_n$, we obtain 
\begin{align*}
d x^i(t)&=\sum_{j=1}^n\left\langle \mathcal A_n h_j, h_i\right\rangle_{H} 
x^j(t)dt+ \sum_{k=1}^m \left\langle \mathcal B_{k,n}, 
h_i\right\rangle_{H}dM_{k}(t)\\ 
&=\sum_{j=1}^n\left\langle \mathcal A h_j, h_i\right\rangle_{H} 
x^j(t)dt+ \sum_{k=1}^m \left\langle \mathcal B_k, 
h_i\right\rangle_{H}dM_{k}(t).
\end{align*}
Hence, the vector of Fourier coefficients $x$ is given by 
\begin{align}\label{galerkwavenowspec}
d x(t)=A x(t)dt+\sum_{k=1}^m b_k dM_k(t),
\end{align}
where $A=\left(\left\langle \mathcal A h_j, 
h_i\right\rangle_{H}\right)_{i, j=1, \ldots, n}=\diag(E_1, \ldots, 
E_{\frac{n}{2}})$ with $E_\ell=\left(\begin{smallmatrix}0 
&\sqrt{\tilde\lambda_\ell}\\ 
-\sqrt{\tilde\lambda_\ell}&-\alpha \end{smallmatrix}\right)$ ($\ell=1, \ldots, 
\frac{n}{2}$), and $\tilde\lambda_\ell$ the eigenvalues of 
$\tilde{\mathcal A}$, and $b_k=\left(\left\langle \mathcal B_k , 
h_i\right\rangle_{H}\right)_{{i=1, \ldots, n}}$ for $k=1, \ldots , m$.
We will often make use of the compact form of the SDE in 
(\ref{galerkwavenowspec}) which is 
\begin{align}\label{galerkwavenowspeccomp}
d x(t)=A x(t)dt+B dM(t),
\end{align}
where $M=\left(M_1, \ldots, M_m\right)^T$ and $B=[b_1, \ldots, b_m]$. 

Applying the spectral Galerkin method to the system (\ref{inrobspeq})-(\ref{introbspoutnew}) the matrices of the semi-discretised system 
(\ref{galerkwavenowspec}) are given by $A=\diag \left(E_1,\ldots, 
E_{\frac{n}{2}}\right)$ with $E_\ell=\left( \begin{smallmatrix} 0 & \ell\\ 
-\ell &-\alpha \end{smallmatrix}\right)$, $b_i=\left(\left\langle \mathcal B_i, 
h_k\right\rangle_{H}\right)_{k=1, \ldots, n}$ with 
\[
\left\langle\mathcal B_i, h_{2 \ell-1}\right\rangle_{H}=0,\;\; \left\langle \mathcal B_i, h_{2 \ell}\right\rangle_{H}=\sqrt{\frac{2}{\pi}} \left\langle 
f_i, \sin(\ell\cdot) 
\right\rangle_{L^2{([0, \pi])}},\] 
and $C=\left[\mathcal 
C h_1, \ldots, \mathcal C h_n\right]$ with 
\begin{align*} 
\mathcal C 
h_{2\ell-1}&=\left(\begin{matrix}\frac{1}{\sqrt{2\pi}\ell^2\epsilon}\left[ 
\cos \left(\ell \left(\frac{\pi}{2} -\epsilon 
\right)\right)-\cos\left(\ell\left(\frac { \pi } { 
2 }+\epsilon \right)\right)\right] & 0 \end{matrix}\right)^T,\\
\mathcal C h_{2 \ell}&=\left(\begin{matrix} 0 & 
\frac{1}{\sqrt{2\pi}\ell\epsilon}\left[ \cos 
\left(\ell \left(\frac{\pi}{2} -\epsilon 
\right)\right)-\cos\left(\ell\left(\frac { \pi } { 
2 }+\epsilon \right)\right)\right]\end{matrix}\right)^T,
\end{align*}
where we assume $n$ to be even and  $\ell=1, \ldots, \frac{n}{2}$.
\begin{figure}[h!]
\centering
\includegraphics[width=0.8\textwidth, height = 5.9cm]{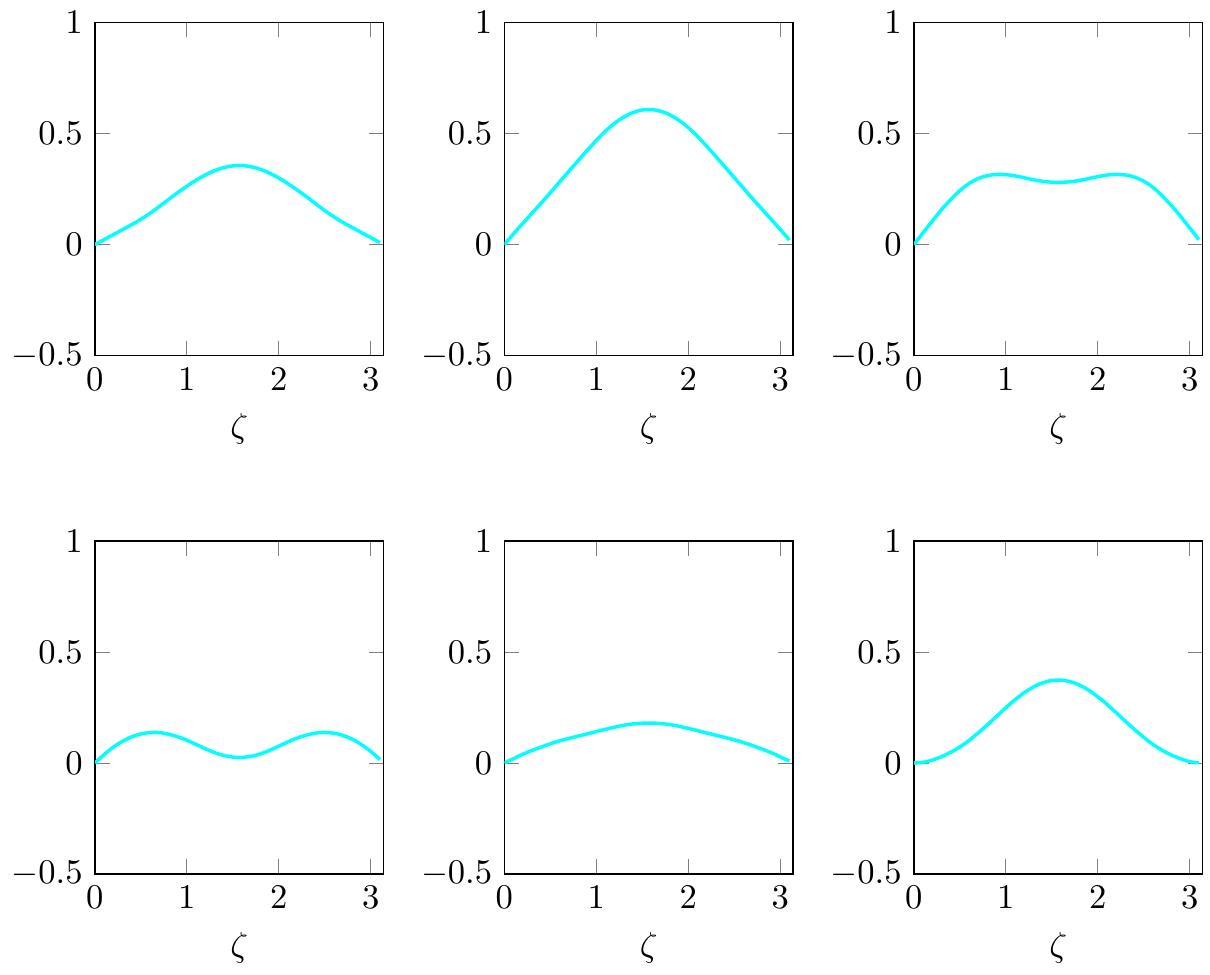}
\caption{Galerkin solution to the stochastic damped wave equation in (\ref{inrobspeq}).}
\label{fig:wavetime}
\end{figure}

In Figure \ref{fig:wavetime} we plot the numerical solution to the stochastic damped wave equation for $\zeta\in[0, \pi]$
and in the time interval $[0,\pi]$ where we set $\alpha = 2$, $n = 1000$ and $m=2$ (e.g. $2$ stochastic inputs). The weighting functions for the
two inputs are $f_1 = 2\exp(-(x-\pi/2)^2)$ and $f_2 = \sin(x)\exp(-(x-\pi/2)^2)$. The noise processes are $M_1(t)=\frac{W(t)}{\sqrt{2}}$ and $M_2(t)=\sum_{i=1}^{N(t)} K_i$ is a compound Poisson process,
where $(N(t))_{t\geq 0}$ is a Poisson process with parameter equal to $1$, $K_i\sim \mathcal{U}(-\sqrt{6}, \sqrt{6})$ are independent uniformly distributed jumps and $W$ is
a standard Wiener process. $M_1$ and $M_2$ are independent. The plot in Figure \ref{fig:wavetime} shows a particular realisation of the solution to (\ref{inrobspeq}) at $6$
specific times $t \in\left\{0.9, 1.2, 1.7, 2.1, 2.46,  2.93\right\}$.
\begin{figure}[h!]
\centering
\includegraphics[width=\textwidth, height = 4cm ]{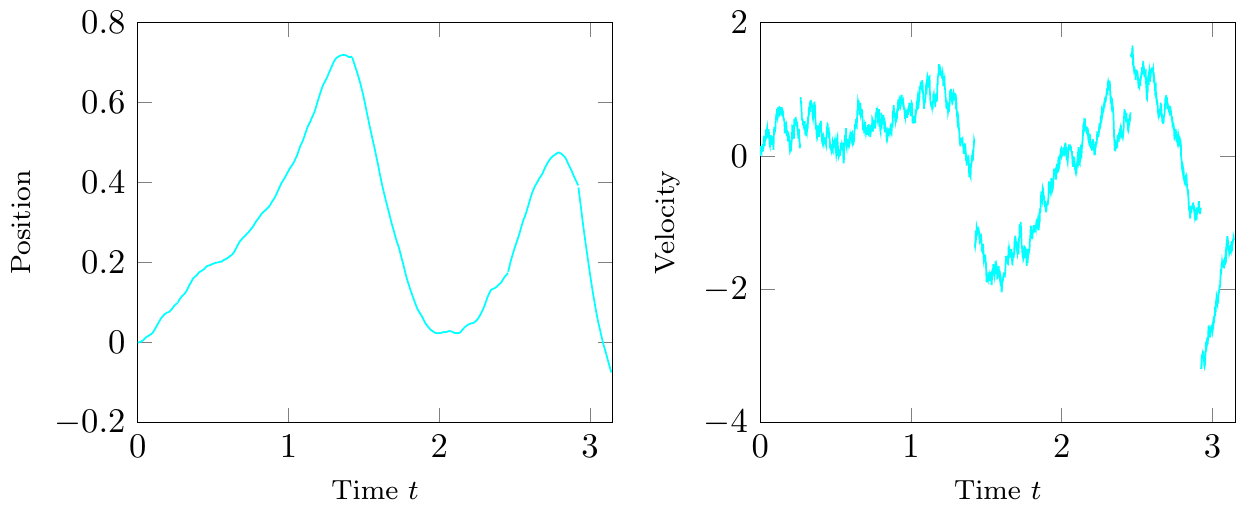}
\caption{Components of the output (\ref{introbspoutnew}) (position and velocity in the middle of the string) of stochastic damped wave equation in (\ref{inrobspeq}).}
\label{fig:position_velocity}
\end{figure}
We see that the string moves up and down as expected due to the nonzero (stochastic) input. We observe that the third snapshot is taken after a jump occured in stochastic process. The corresponding output, namely both the position and the velocity in the middle of the string, is shown in Figure \ref{fig:position_velocity}. In the plot for the velocity the noise generated by the L\'evy process can be seen. The trajectory of the velocity is impacted by L\'evy noise with jumps, where the velocity (e.g. the impact by wind) is randomly increased or reduced.
\begin{figure}[h!]
\centering
\includegraphics[scale=0.9]{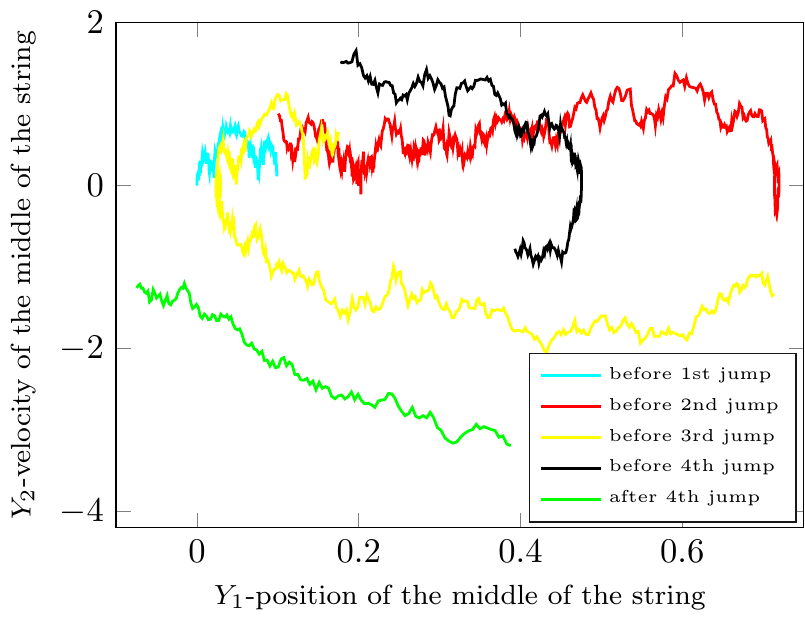}
\caption{Output of stochastic damped wave equation in (\ref{inrobspeq})-(\ref{introbspoutnew}) in the phase plane.}
\label{fig:paths}
\end{figure}
The trajectory for the position of the cable in Figure  \ref{fig:position_velocity} is smoother as it is the integral of the velocity. Finally, Figure \ref{fig:paths} shows the velocity versus the position of the string, for the same sample path, in a phase portrait. The four jumps are clearly visible. 

\section{Numerical examples for MOR}
\label{sec:numerics}

We consider the spectral Galerkin discretisation of the second order damped wave equation  which we discussed in detail in Section \ref{sec:wave}, and in
particular, the example in (\ref{inrobspeq})-(\ref{introbspoutnew}) with two stochastic inputs and two outputs, namely position and velocity of the 
middle of the string. We set $\alpha=2$ and choose the weighting functions $f_i$ and the noise processes $M_i$ ($i=1, 2$) as in 
Figure \ref{fig:wavetime}. We fix the state dimension to $n=1000$ and reduce the Galerkin solution by BT and SPA. For computing the trajectories of 
the SDE we use the Euler-Maruyama method (see, e.g. \cite{highamcompkloeden,highamcompeuler}).
\begin{figure}[h!]
\centering
\includegraphics[width=\textwidth]{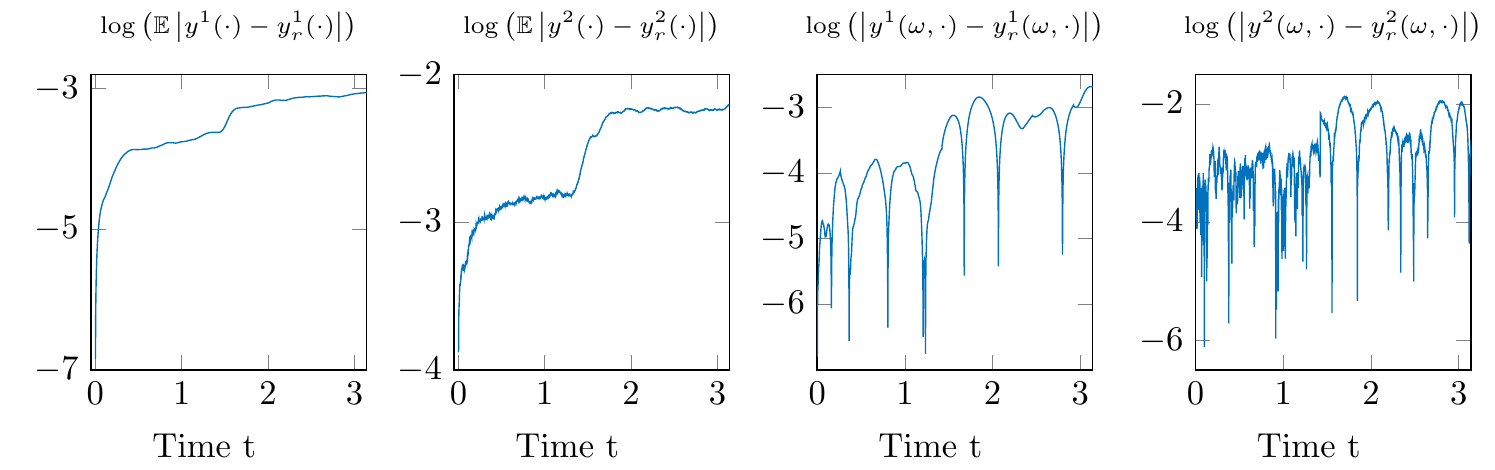}
\caption{Logarithmic errors of BT for position $y^1$ and velocity $y^2$ with $r=6$.}
\label{fig:BTr6}
\end{figure}
Figures \ref{fig:BTr6} and \ref{fig:BTr24} show the logarithmic errors for the position $y^1$ and the velocity $y^2$ of the middle of the string,
if MOR by BT is applied to the wave equation with stochastic inputs when reduced models of dimension $6$ and $24$, respectively, are computed. 
\begin{figure}[h!]
\centering
\includegraphics[width=\textwidth]{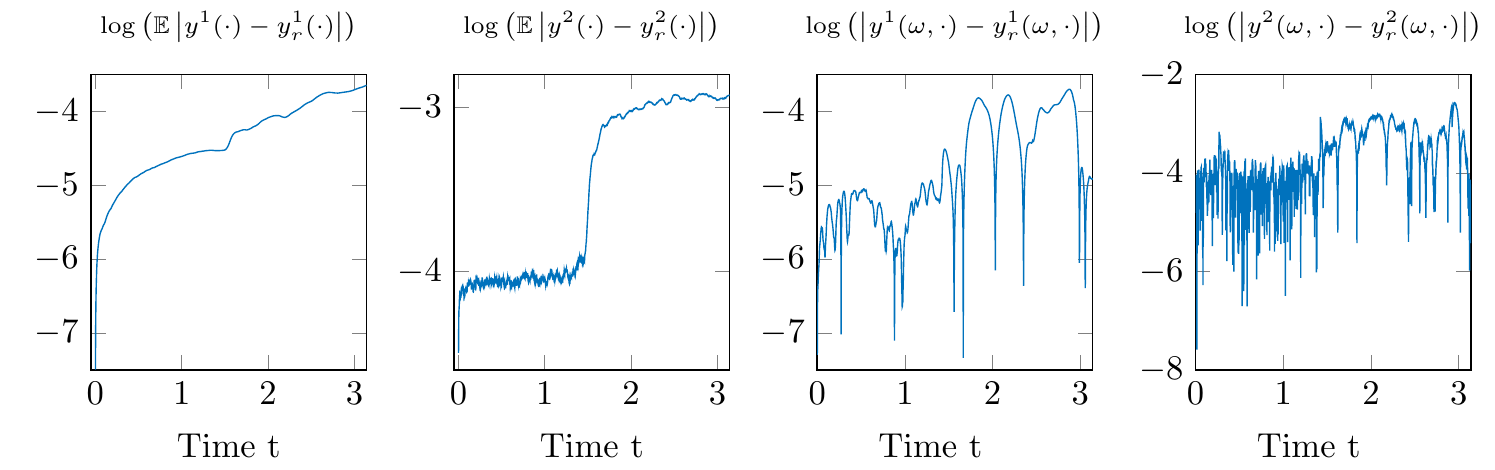}
\caption{Logarithmic errors of BT for position $y^1$ and velocity $y^2$ with $r=24$.}
\label{fig:BTr24}
\end{figure}
The first two plots in each of the figures show the logarithmic mean error for both the position and the velocity. One observation is that the 
position is generally more accurate than the velocity (about one order of magnitude here), since the trajectories are smoother. Moreover, comparing 
the expected values of the errors of the reduced model of dimension $r=6$ (first two plots in Figure \ref{fig:BTr6}) with the one of dimension $r=24$ 
(first two plots in Figure \ref{fig:BTr24}) it can be seen that the latter ones are more accurate (an improvement of about one order of magnitude) as 
one would expect. The last two plots in Figures \ref{fig:BTr6} and \ref{fig:BTr24} show the logarithmic errors for position and velocity for one 
particular trajectory, which is the same as the one for the sample we considered in Section \ref{sec:wave}.
\begin{figure}[h!]
\centering
\includegraphics[width=\textwidth]{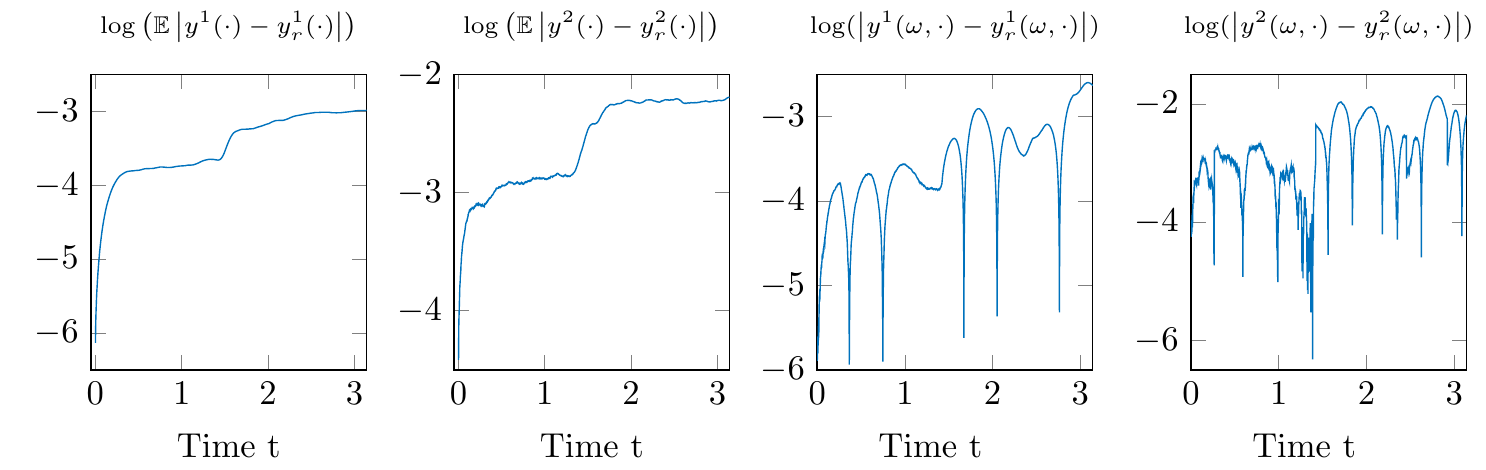}
\caption{Logarithmic errors of SPA for position $y^1$ and velocity $y^2$ with $r=6$.}
\label{fig:SPAr6}
\end{figure}
Figures \ref{fig:SPAr6} and \ref{fig:SPAr24} show the logarithmic errors for the position $y^1$ and the velocity $y^2$ of the middle of the string, 
if MOR by SPA is applied to the wave equation with stochastic inputs when reduced models of dimension $6$ and $24$, respectively, are computed. 
Again, the first two plots show the mean errors while the last two plots show the errors in particular trajectories.
\begin{figure}[h!]
\centering
\includegraphics[width=\textwidth]{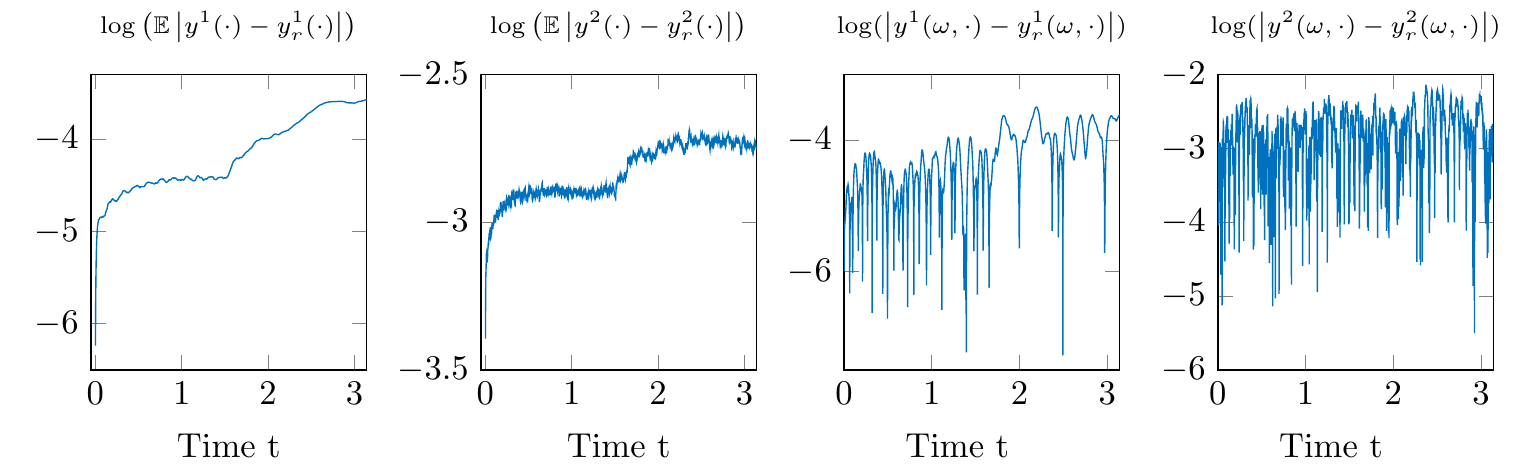}
\caption{Logarithmic errors of SPA for position $y^1$ and velocity $y^2$ with $r=24$.}
\label{fig:SPAr24}
\end{figure}
We observe that the error in the position is smaller than the error in the velocity, and, the error is smaller if a larger dimension of the reduced order model is used.

Finally, we compare the error bounds for BT (see Theorem \ref{bterrorbound}) and SPA (see Theorem \ref{spaerrorbound}) with the worst case mean 
errors, that is 
\begin{align*}  
\sup_{t\in[0, \pi]} \mathbb E\left\|y(t)-y_r(t)\right\|_{\mathbb{R}^p}
   \end{align*}
for both methods in Table \ref{tab:errorbd}, where $y=\left(\begin{matrix} y^1&y^ 2\end{matrix}\right)^T$ is the full output of the original model 
and $y_r$ the ROM output.
\begin{table}[h!]
\begin{center}
\begin{tabular}{|c|c|c|c|c|c|}
\hline
\text{Dim. ROM}&\text{Error BT }&\text{Error bound BT }&\text{Error SPA } 
&\text{Error bound SPA }
\\
\hline
\hline
   2 &  7.6387e-02 &  9.3245e-02 &  1.0852e-01 &  1.2293e-01\\
   4 &  8.5160e-03 &  1.2180e-02 &  8.6050e-03 &  1.2185e-02\\
   8 &  5.1560e-03 &  9.6638e-03 &  5.6720e-03 &  9.7072e-03\\
   16 &  1.8570e-03 &  6.6764e-03 &  2.4970e-03 &  6.7382e-03\\
   32 & 6.7050e-04  & 4.3849e-03  & 1.4410e-03 &  4.9106e-03\\
   64 &  9.9130e-05 &  2.3491e-03 &  3.1440e-04 &  2.6354e-03\\
\hline
\end{tabular}
\caption{Error and error bounds for both BT and SPA and several dimensions of the reduced order model (ROM).}
\label{tab:errorbd}
\end{center}
\end{table}
First, as expected both mean errors and error bounds are getting smaller the larger the size of the ROM. Moreover, both error bounds 
are rather tight and close to the actual error of the ROM, e.g. the bounds, which are worst case bounds also provide a good prediction of the true time domain error. We also note that BT performs better than SPA, both in actually 
computed mean  errors as well as in terms of the error bounds.  

\section{Conclusions}

We have presented theory for balancing related model order reduction (MOR) applied to linear stochastic differential equations (SDEs) with additive L\'evy noise. In particular we extended the concepts of reachability and observability to stochastic systems and formulated a new reachability Gramian. We then showed how balancing related MOR which is well known for deterministic systems can be extended to SDEs with additive L\'evy noise, e.g. leads to the solution of a Lyapunov equation (with a slightly different right hand side). We proved a general error bound for reduced (asymptotically stable) systems in this setting and then gave specific bounds for balanced truncation (BT) and singular perturbation approximation (SPA) which depended on the neglected (small) Hankel singular values of the linear system. We finally applied our theory to a second order damped wave equation, discretised using a spectral Galerkin method, and controlled by L\'evy noise. The numerical results showed that MOR can be applied successfully and that errors for both BT and SPA are small, and the error bounds tight.

\appendix

\section{Ito calculus}
Let all stochastic processes appearing in this section be defined on a filtered 
probability space $\left(\Omega, \mathcal F, (\mathcal F_t)_{t\geq 0}, 
\mathbb P\right)$\footnote{$(\mathcal F_t)_{t\geq 0}$ shall be right continuous 
and complete.}. We denote the set of all c\`adl\`ag square integrable $\mathbb 
R$-valued martingales with respect to $(\mathcal F_t)_{t\geq 0}$ by $\mathcal 
M^2(\mathbb R)$.\medskip

Let $Z_1, Z_2$ be scalar semimartingales. We set $\Delta 
Z_i(s):=Z_i(s)-Z_i(s-)$ with $Z_i(s-):=\lim_{t\uparrow s} 
Z_i(t)$ for $i=1, 2$. Then the Ito product formula 
\begin{align}\label{profriot}
Z_1(t) Z_2(t)=Z_1(0) Z_2(0)+\int_0^t Z_1(s-)dZ_2(s)+\int_0^t 
Z_2(s-)dZ_1(s)+[Z_1, Z_2]_t
\end{align}
for $t\geq 0$ holds, see \cite{semimartingalesmichel} or 
\cite{applebaumendlich} for the special case of L\'evy-type integrals.  
By \cite[Theorem 4.52]{limittheorems}, the compensator 
process $[Z_1, Z_2]$ is given by \begin{align}\label{decomqucov}
[Z_1, Z_2]_t=\left\langle M_1^c, M_2^c\right\rangle_t+\sum_{0\leq s\leq t} 
\Delta Z_1(s) \Delta Z_2(s)
\end{align}
for $t\geq 0$, where $M_1^c$, $M_2^c\in\mathcal M^2(\mathbb R)$ are the 
continuous martingale parts of $Z_1$ and $Z_2$ (cf. \cite[Theorem 4.18]{limittheorems}). The process $\left\langle M_1^c, M_2^c\right\rangle$ is a 
uniquely defined angle bracket process that ensures that $M_1^c 
M_2^c-\left\langle M_1^c, M_2^c\right\rangle$ is an $(\mathcal F_t)_{t\geq 0}$- 
martingale, see \cite[Proposition 17.2]{semimartingalesmichel}. As a simple 
consequence of (\ref{profriot}), we have:
\begin{kor}\label{iotprodformelmatpro}
Let $Y$ be an $\mathbb R^d$-valued and $Z$ be an $\mathbb R^n$-valued 
semimartingale, then we have\begin{align*}
Y(t) Z^T(t)=Y(0) Z^T(0)+\int_0^t dY(s) Z^T(s-) +\int_0^t Y(s-) 
dZ^T(s)+\left([Y_i,Z_j]_t\right)_{{i=1, \ldots, d}\atop {j=1, \ldots, n}}
\end{align*}
for all $t\geq 0$.
\begin{proof}
Considering the stochastic differential of the $ij$-th component of the 
matrix-valued process $Y(t) Z^T(t)$, $t\geq 0$, and using (\ref{profriot}) gives the result, see also \cite{redmannbenner}.
\end{proof}
\end{kor}

\section{Proof of Theorem \ref{th:mildsol}}
\label{sec:app2}

Using $\left\|\sum_{k=1}^q a_k\right\|_{H}^2\leq q\sum_{k=1}^q \left\|
a_k\right\|_{H}^2$ for $a_k\in H$, we obtain
{\allowdisplaybreaks \begin{align*}
\mathbb E\left\|\mathcal X(t)-X_n(t)\right\|_{H}^2&\leq 2\mathbb E 
\left\|S(t)X_0-S_n(t)X_{0, n}\right\|_{H}^2\\ &\ \ \ +2m\sum_{k=1}^m\mathbb E 
\left\|\int_0^t (S(t-s) \mathcal B_k-S_n(t-s) \mathcal 
B_{k, n}) dM_k(s)\right\|_{H}^2.
\end{align*}}
Ito's isometry (see e.g. \cite{zabczyk}) yields that the right hand side can be bounded by $2\mathbb E 
\left\|S(t)X_0-S_n(t)X_{0, n}\right\|_{H}^2+2m\sum_{k=1}^m 
\int_0^t \left\|(S(t-s) \mathcal B_k-S_n(t-s) \mathcal 
B_{k, n})\right\|_{H}^2ds\; \mathbb E\left[M^2_k(1)\right]$. Since $(S(t))_{t\geq 0}$ is a contraction semigroup, we 
have\begin{align}\nonumber
\mathbb E \left\|S(t)X_0-S_n(t)X_{0, n}\right\|_{H}^2&\leq 2 
\mathbb E \left\|S(t)X_0-S_n(t)X_{0}\right\|_{H}^2+2 \mathbb E 
\left\|S_n(t)X_{0}-S_n(t)X_{0, n}\right\|_{H}^2\\ 
\label{inconcon2} &\leq 2 \mathbb E \left\|S(t)X_{0}-S_n(t)X_{0}\right\|_{H}^2+ 
2 \mathbb E \left\|X_0-X_{0, n}\right\|_{H}^2.
\end{align}
By the representation $S_n(t)x=\sum_{i=1}^n \left\langle S(t) x, 
h_i\right\rangle_{H} h_i$ ($x\in H$) and Lebesgue's 
theorem, the bound in (\ref{inconcon2}) tends to zero for $n\rightarrow \infty$. For $k\in\left\{1, \ldots, m\right\}$ we get
\begin{align*}
&\left\|S(t-s) \mathcal B_k-S_n(t-s) \mathcal B_{k, n}\right\|_{H}^2
\\&\leq 2\left\|S(t-s) \mathcal B_k-S_n(t-s) \mathcal B_k\right\|_{H}^2+2 
\left\|S_n(t-s) \mathcal B_k-S_n(t-s) \mathcal B_{k, n}\right\|_{H}^2
\\&\leq 2 
\left\|S(t-s) \mathcal B_k-S_n(t-s) \mathcal B_k\right\|_{H}^2+2\left\| 
\mathcal B_k -\mathcal B_{k, n}\right\|_{H}^2
\end{align*}
which tends to zero for $n\rightarrow \infty$ and hence 
\begin{align*}
\sum_{k=1}^m 
\int_0^t \left\|(S(t-s) \mathcal B_k-S_n(t-s) \mathcal 
B_{k, n})\right\|_{H}^2ds\; \mathbb E\left[M^2_k(1)\right]\rightarrow 0
\end{align*}
for $n\rightarrow \infty$ by Lebesgue's theorem. 

\bibliographystyle{abbrv}

\end{document}